\documentclass{article}

\title{Extended Circular Nim}
\author{Koki Suetsugu\footnote{Toyo University, Waseda University, and Osaka Metropolitan University, suetsugu.koki@gmail.com}}
\date{}
\usepackage{color}

\usepackage[margin=20truemm]{geometry}
\usepackage{graphicx}
\usepackage{latexsym}
\usepackage{amsmath, amsthm, amssymb}
\usepackage{here}

\newtheorem{theorem}{Theorem}
\newtheorem{definition}{Definition}
\newtheorem{example}{Example}
\newtheorem{corollary}{Corollary}
\newtheorem{observation}{Observation}

\begin{document}

\maketitle
\begin{abstract}
{\sc Circular nim} ${\rm CN}(m, k)$ is a variant of {\sc nim}, in which there are $m$ piles of tokens arranged in a circle and each player, in their turn, chooses at most $k$ consecutive piles in the circle and removes an arbitrary number of tokens from each pile. The player must remove at least one token in total.
For some cases of $m$ and $k$, closed formulas to determine which player has a winning strategy have been found. Almost all cases are still open problems.
In this paper, we consider a variant of {\sc circular nim,} {\sc extended circular nim}. In {\sc extended circular nim} ${\rm ECN}(m_S, k)$, there are $m$ piles of tokens arranged in a circle. $S$ is a set of positive integers less than or equal to half of $m$.
In each turn, a player chooses an integer $s \in S$.
Then the player selects at most $k$ piles among those located every $s$-th position on the circle,
and removes an arbitrary number of tokens from each selected pile.
We show some closed formulas to determine which player has a winning strategy for the cases where the number of piles is no more than eight, and for a few generalized cases.
\end{abstract}

\section{Introduction}
In this paper, we consider a generalization of a ruleset, {\sc circular nim}. In {\sc circular nim} ${\rm CN}(m, k)$, there are $m$ piles of tokens arranged in a circle. Each player, in their turn, chooses at most $k$ consecutive piles in the circle and removes an arbitrary nonnegative number of tokens from each pile. The player must remove at least one token in total. This game has been studied in some articles, and winning positions for the previous player are characterized for some cases of $(m, k)$. Figure \ref{fig:circularexample} shows an example of a move in {\sc circular nim} ${\rm CN}(6, 3)$.

\begin{figure}
\centering
    \begin{minipage}{0.44\linewidth}
        \centering
    \includegraphics[width=\linewidth]{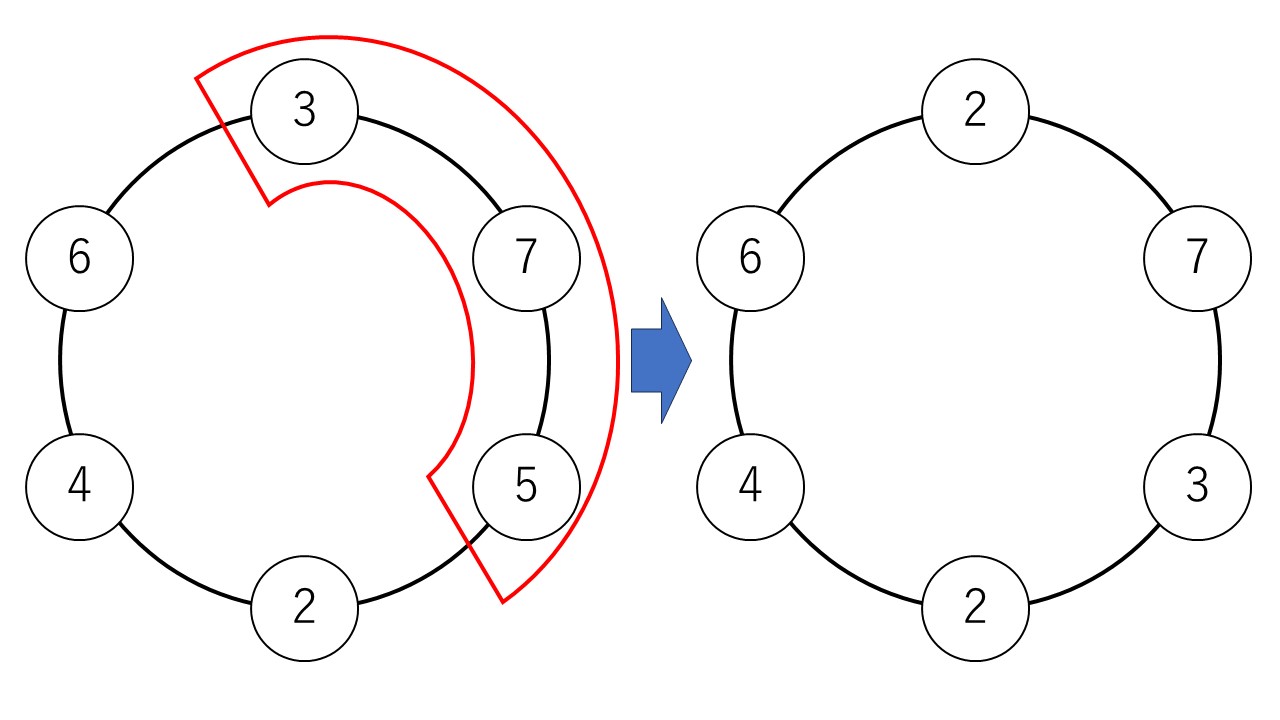}
    \caption{Move in {\sc circular nim} ${\rm CN}(6, 3)$.}
    \label{fig:circularexample}    
    \end{minipage}
    \begin{minipage}{0.1\linewidth}
    \end{minipage}
    \begin{minipage}{0.44\linewidth}
        \centering
    \includegraphics[width=\linewidth]{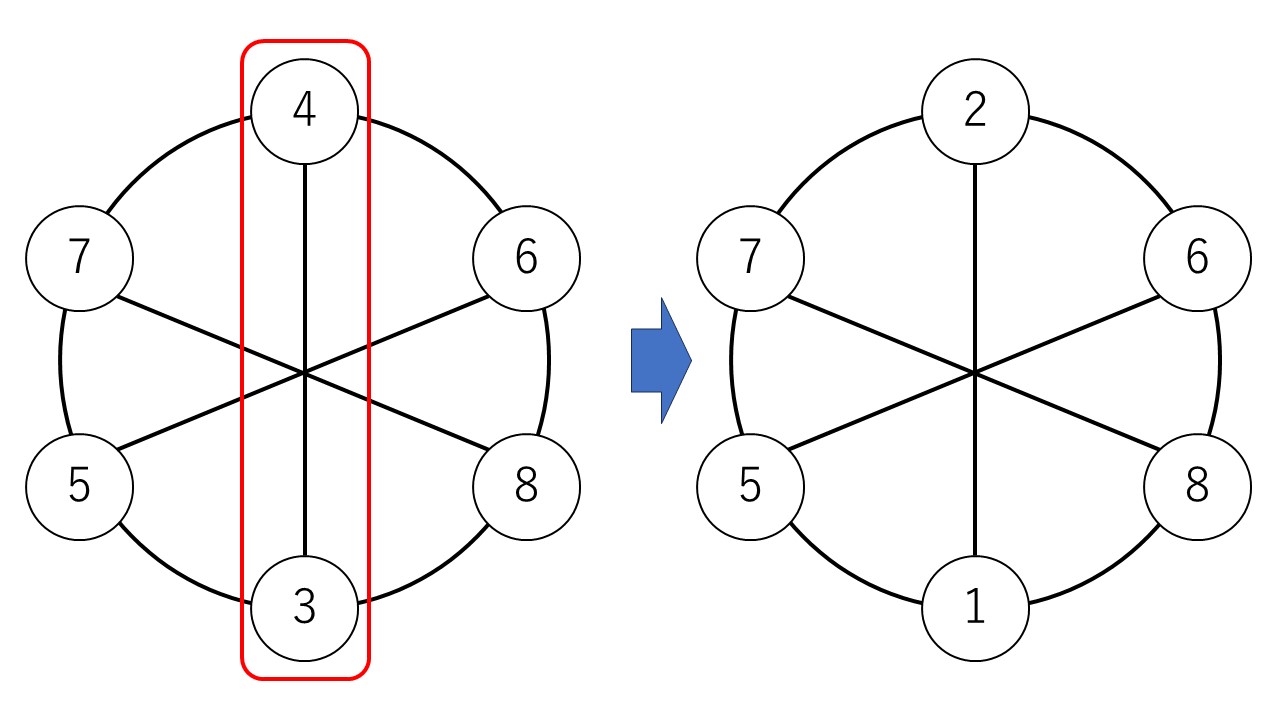}
    \caption{Move in {\sc extended circular nim} ${\rm ECN}(6_{\{1, 3\}}, 2)$.}
    \label{fig:ecircularexample}    
    \end{minipage}
\end{figure}

In this paper, we consider {\sc extended circular nim} ${\rm ECN}(m_S, k)$, a generalization of {\sc circular nim}. In ${\rm ECN}(m_S, k)$, there are $m$ piles of tokens arranged in a circle. $S$ is a set of positive integers less than or equal to half of $m$.
Each player, in their turn, chooses at most $k$ piles selected every $s$-th pile in a circle for an $s \in S$.
That is, ${\rm ECN}(m_{\{1\}}, k)$ is the same as ${\rm CN}(m, k).$
Figure \ref{fig:ecircularexample} shows an example of a  move in {\sc extended circular nim} ${\rm ECN}(6_{\{1, 3\}}, 2)$.

\subsection{Combinatorial game theory}
In order to study games such as {\sc extended circular nim}, we need to introduce some background of combinatorial game theory (CGT).
Combinatorial game theory studies two-player perfect information games with no chance. In this paper, we use some theories of {\em impartial rulesets}, that is, in every position, both players have the same sets of options. We also assume that the games are {\em short}, that is, the game must end in a finite number of moves, and under {\em normal play convention}, that is, the player who cannot move is the loser.

With these assumptions, the player who has a winning strategy is uniquely determined for each position.
A position in which the next player has a winning strategy is called an $\mathcal{N}$-position, and a position in which the previous player has a winning strategy is called a $\mathcal{P}$-position.

It is well-known that if all positions of an impartial ruleset are separated into two sets $P$ and $N$, which satisfy the following properties, then $P$ and $N$ are the sets of all $\mathcal{P}$-positions and all $\mathcal{N}$-positions, respectively:  

\begin{itemize}
\item All terminal positions are in $P$.
\item For every position in $P$, all options of the position are in $N$.
\item For every position in $N$, at least one option of the position is in $P$.
\end{itemize}

We will use this fact for some proofs. For the details of combinatorial game theory, see \cite{Sie13}.



\subsection{Some impartial rulesets}
\subsubsection{{\sc Nim} and {\sc Moore's nim}}
In early studies, the characterization of $\mathcal{P}$-positions is considered in a simple formula.

As examples of rulesets in which the sets of $\mathcal{P}$-positions are characterized, the following two rulesets, {\sc nim} and {\sc Moore's nim} are well-known.

\begin{definition}
In {\sc nim}, there are some piles of tokens.
The player, in their turn, chooses one pile and removes an arbitrary positive number of tokens from the pile.
\end{definition}

Note that since we assume that the winner is determined by normal play convention, the player who removes the last token is the winner.

A position in {\sc nim} is denoted as $(n_0, n_1, \ldots, n_{m-1})$ when there are $m$ piles and each pile has $n_0, n_1, \ldots, n_{m-1}$ tokens. 
Let $\oplus$ be the exclusive OR operator for binary notation.
By using the following theorem, we can easily determine which player has a winning strategy in {\sc nim}.

\begin{theorem}[\cite{Bou01}]
\label{thm_nim}
A position $(n_0, n_1, \ldots, n_{m-1})$ in {\sc nim} is a $\mathcal{P}$-position if and only if $n_0\oplus n_1 \oplus \cdots \oplus n_{m-1} = 0$.
\end{theorem}

{\sc Moore's nim} is a generalization of {\sc nim}.

\begin{definition}
In {\sc Moore's nim} ${\rm MN}(m, k)$, there are $m$ piles of tokens. Each player, in their turn, chooses at most $k$ piles and removes an arbitrary number of tokens from each pile. The player must remove at least one token in total.
\end{definition}

When $k=1,$ {\sc Moore's nim} is the original {\sc nim}. Thus, {\sc Moore's nim} is a generalization of {\sc nim}. The $\mathcal{P}$-positions of this ruleset are characterized by the following theorem, which generalizes Theorem \ref{thm_nim}.

\begin{theorem}
[\cite{Moo10}]
\label{thm_moo}
A position $(n_0, n_1, \ldots, n_{m-1})$ in {\sc Moore's nim} ${\rm MN}(m, k)$ is a $\mathcal{P}$-position if and
only if when each $n_i$ is expanded in the binary form, the summand of each digit contains a multiple of $(k + 1)$ number of $1$s.
\end{theorem}

Note that when $k = m - 1$, a position $(n_0, n_1, \ldots, n_{m-1})$ in {\sc Moore's nim} ${\rm MN}(m, m-1)$ is a $\mathcal{P}$-position   if and only if $n_0 = n_1 = \cdots = n_{m-1}$.
\subsubsection{Simplicial nim}

{\sc Simplicial nim} is a generalization of {\sc nim} and {\sc Moore's nim}.

Let $V$ be a finite set of vertices.
A {\em simplicial complex} $\Delta$ on $V$ is a subset of $2^{V}$
 which satisfies the following:  
 \begin{itemize}
     \item For every element $v \in V,$ $\{v\} \in \Delta$
     \item If $F \in \Delta$ and $G \subseteq F,$ then $G \in \Delta$.
 \end{itemize}
 An element of $\Delta$ is termed a {\em face} of $\Delta$.
Ehrenborg and Steingr\'{i}msson studied {\sc nim} on simplicial complexes or {\sc simplicial nim} in \cite{ES96}. Let $\Delta$ be a simplicial complex on a finite set $V$. Each vertex has several tokens. In {\sc nim} on $\Delta$, the players, in turn, choose a non-empty face $F$ in $\Delta$ and arbitrarily remove any positive number of tokens from all vertices in $F$. The original {\sc nim} and {\sc moore's nim}
can be considered as special cases of this game.
Ehrenborg and Steingr\'{i}msson did not characterize the $\mathcal{P}$-positions of {\sc nim} on $\Delta$ without restrictions but discovered good constructions and characterized $\mathcal{P}$-positions for certain cases. For example, they proved that circuits, which are subsets of the vertex set and not faces of the simplicial complex, but all proper subsets of them are faces, can be used for characterizing $\mathcal{P}$-positions.
 Horrocks \cite{Hor10} and Penn \cite{Pen21} also studied this ruleset. {\sc Hypergraph nim}, studied in \cite{BGHMM19a, BGHMM19b, BGHMM23}, is known as a generalization of {\sc simplicial nim}.

For nonnegative integers $n_0, n_1, \ldots, n_{m-1},$ a position $(n_0, n_1, \ldots, n_{m-1})$ of {\sc nim} on $\Delta$ is a position such that there are $n_i$ tokens on $v_i$ for each $i$.

\begin{example}
Let $V = (v_0, v_1, v_2)$ and $\Delta = \{ \{v_0\}, \{v_1\}, \{v_2\}, \{v_0, v_1\}\}.$
Consider the position $A = (3, 4, 5)$ of {\sc nim} on $\Delta$.
In this case, the position $B = (1, 2, 5)$ is an option of $A$ since $\{v_0, v_1\}$ is a face of $\Delta$ but the position $C = (3, 3, 3)$ is not an option of $A$ since $\{v_1, v_2\}$ is not a face of $\Delta$. Figure \ref{fig:simlicialexample} shows these legal and illegal moves.
\end{example}

\begin{figure}[tb]
\centering
\includegraphics[width = 10cm]{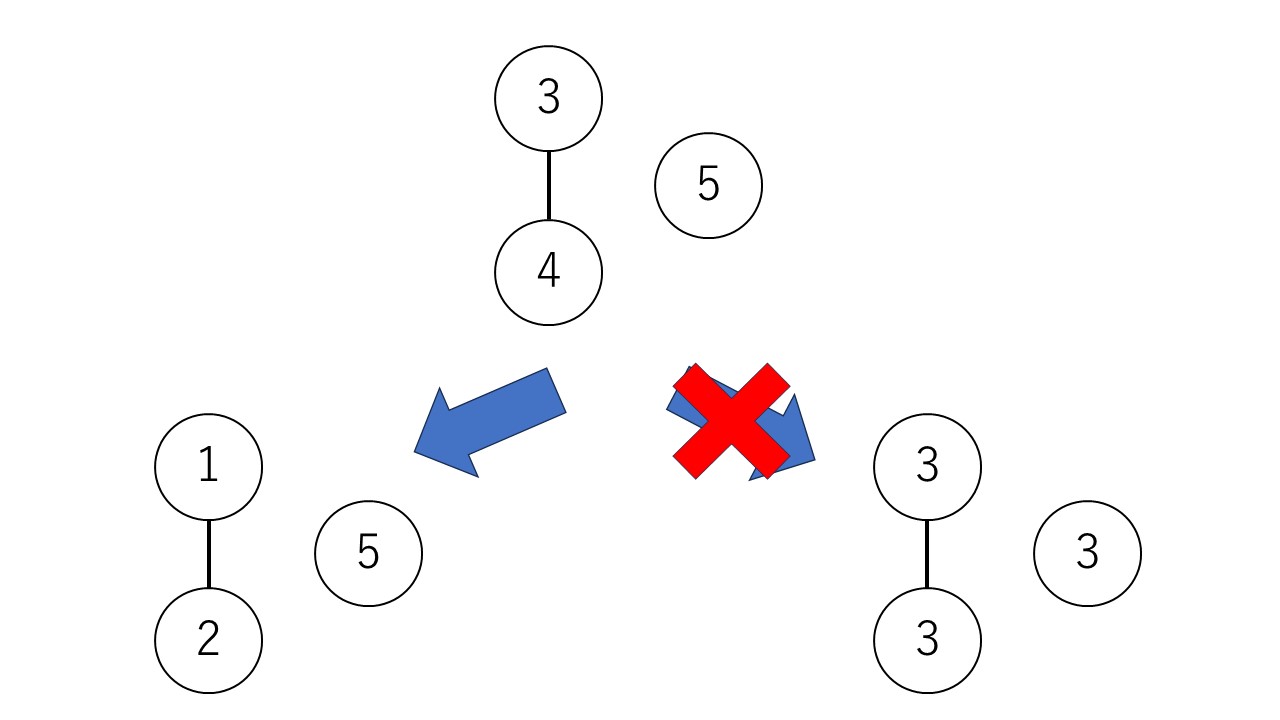}
\caption{A legal move and an illegal move in {\sc simplicial nim}.}
\label{fig:simlicialexample}
\end{figure}

For the later discussion, we refer to the following Definitions \ref{def:vector}, \ref{def:nimregular} and Theorem \ref{thm:nimregular} from \cite{ES96}.

\begin{definition}
\label{def:vector}
    Given a complex $\Delta$ with vertex set $V$, let $\mathbb{N}^V$ be the set of vectors indexed by $V$ whose entries are non-negative integers. Let ${\bf e}(v)$ be the $v$-th unit vector, that is, the vector whose $v$-th entry is $1$ and all other entries are $0$. For a subset $A$ of $V$, let 
    $$
    {\bf e}(A) = \sum_{v \in A} {\bf e}(v).
    $$
\end{definition}

Note that using this notation, a position $G = (n_0, \ldots, n_{m-1})$ can be written as 
$$
G = \sum_{0 \le i < |V|} n_i \cdot {\bf e}(v_i).
$$

\begin{definition}
\label{def:nimregular}
    Let $\Delta$ be a simplicial complex and $\mathcal{B}$ a collection of subsets of $\Delta$ such that the following Conditions (A), (B), and (C) are satisfied. Here, $A = B \sqcup C$ means $A = B \cup C$ and $B \cap C = \emptyset.$

        \begin{itemize}
        \item[(A)] The empty set belongs to $\mathcal{B}$.
        
        \item[(B)] Suppose that $F$ is a face of $\Delta$, that $B, B' \in \mathcal{B}$ and that $B' = F \sqcup B$. Then $F$ is the empty face.
        
        \item[(C)] Let $F$ be a face of $\Delta$ and let $S$ be a subset of $V$. Then there exist faces $K$ and $G$  of $\Delta$, with $K \subseteq F \subseteq G$, such that $G - F \subseteq S$ and $(S - G) \sqcup K \in \mathcal{B}$.
        \end{itemize}

    Then $\mathcal{B}$ is said to be a {\sc nim}-basis for $\Delta$. A simplicial complex which has a {\sc nim}-basis is said to be {\sc nim}-regular.
\end{definition}
\begin{theorem}
\label{thm:nimregular}
    Assume that $\Delta$ is a {\sc nim}-regular simplicial complex with {\sc nim}-basis $\mathcal{B}$. Then, a position $G$ is a $\mathcal{P}$-position on $\Delta$ if and only if 
    $$
    G = \sum_{i \ge 0}2^i \cdot {\bf e}(A_i),
    $$
    where $A_i$ belongs to $\mathcal{B}$ for all $i \ge 0$.
\end{theorem}

\subsubsection{{\sc Circular nim}}
Now we can define {\sc circular nim} as a special case of {\sc simplicial nim}.

\begin{definition}
Let $m$ and $k$ be positive integers with $k \leq m$.
{\sc Circular nim} ${\rm CN}(m, k)$ is {\sc nim} on a simplicial complex $\Delta$, where $V = \{v_0, \ldots, v_{m-1}\}$ and $\Delta$ is the set of all non-empty subset of  $\bigcup_{i = 0}^{m - 1}  \{v_i, v_{((i+1) \bmod m)}, \ldots, v_{((i+k-1) \bmod m)}\}$. 
\end{definition}
That is, in {\sc circula nim}, 
there are $m$ piles of tokens arranged in a circle and each player, in their turn, chooses at most $k$ consecutive piles in the circle and removes an arbitrary number of tokens from each pile, where in total at least one token must be removed\footnote{Note that in {\sc circular nim} and {\sc extended circular nim} defined in the next section, a pile is considered to be still existing when counting $k$ consecutive piles even if all tokens on the pile have been removed. In contrast, in {\sc shrinking circular nim}, a variant of {\sc circular nim} discussed in \cite{OS24}, an empty pile is ignored.}.
When $k = 1$, ${\rm CN}(m,k)$ is the same ruleset as $m$-pile {\sc nim}. 
Also, when $k = m - 1$, ${\rm CN}(m,k)$ is the same ruleset as {\sc Moore's nim} ${\rm MN}(m, m-1)$.
Furthermore, when $k = m$, the player can select all of the piles; thus, we can consider this ruleset as a single-pile {\sc nim}. Therefore, these cases are trivial and other cases are studied in \cite{ES96, Duf96, Hor10, DH13, DHV22} as follows. For any $M = (n_0, \ldots, n_{m-1})$, we say $ M \in_\circlearrowleft P$ if there exists $i (< m)$ such that $(n_i, n_{((i + 1) \bmod m)}, \ldots, n_{((i + m - 1) \bmod m)}) \in P$ or $(n_i, n_{((i + m - 1) \bmod m)}, \ldots, n_{((i + 1) \bmod m)}) \in P$. That is, if a position in $P$ is equal to $M$, or it is the same position as $M$ after rotating or flipping it, then we say $ M \in_\circlearrowleft P$.  

\begin{itemize} 
\item Consider a position $M = (n_0,n_1,n_2,n_3)$ in ${\rm CN}(4,2)$. $M$ is a $\mathcal{P}$-position if and only if $n_0 = n_2$ and $n_1 = n_3$.

\item Consider a position $M = (n_0, n_1, n_2, n_3, n_4)$ in ${\rm CN}(5,2)$. 
$M$ is a $\mathcal{P}$-position if and only if $M \in_\circlearrowleft P =  \{(n_0, n_1, n_2, n_3, n_4) \mid n_0 = \max(M), n_0 + n_1 = n_2 + n_3, n_1 = n_4\}$.

\item Consider a position $M = (n_0, n_1, n_2, n_3, n_4)$ in ${\rm CN}(5,3)$. $M$
 is a $\mathcal{P}$-position if and only if $M \in_\circlearrowleft P = \{(n_0, n_1, n_2, n_3, n_4) \mid  n_0 = 0 , n_1 = n_2 + n_3= n_4\}.$
 
\item Consider a position $M = (n_0, n_1, n_2, n_3, n_4, n_5)$ in ${\rm CN}(6,3)$. $M$ is a $\mathcal{P}$-position if and only if  $M \in_\circlearrowleft P = \{(n_0, n_1, n_2, n_3, n_4, n_5) \mid n_0 + n_1  = n_3 + n_4, n_1 +n_2 = n_4 + n_5\}$.

\item Consider a position $M = (n_0, n_1, n_2, n_3, n_4, n_5)$ in ${\rm CN}(6,4)$.  $M$ is a $\mathcal{P}$-position if and only if $M \in_\circlearrowleft P = \{(n_0, n_1, n_2, n_3, n_4, n_5) \mid n_0 = \min(M), n_0 + n_1  = n_3 + n_4, n_1 + n_2 = n_4 + n_5, n_0 \oplus n_2 \oplus n_4 =~0 \}$.

\item Consider a position $M = (n_0,n_1,n_2,n_3,n_4, n_5,n_6)$ in ${\rm CN}(7,4)$. $M$ is a $\mathcal{P}$-position if and only if $M \in_\circlearrowleft P = P_1 \cup P_2 \cup P_3 \cup P_4$, where 

\begin{eqnarray*}
    P_1 &=& \{(n_0, n_1, n_2, n_3, n_4, n_5, n_6) \mid   n_0 = n_1 = 0, 
    n_3+n_4+n_5 = n_2 = n_6 > 0\},\\
    P_2 &=& \{(n_0, n_1, n_2, n_3, n_4, n_5, n_6) \mid n_0=n_1=n_2=n_3=n_4=n_5=n_6\},\\
    P_3 &=& \{(n_0, n_1, n_2, n_3, n_4, n_5, n_6) \mid n_0 = n_1, n_2 = n_6, n_3 = n_5, n_0 + n_2 = n_3 + n_4, \\ & & 0 < n_0 < n_4, n_0 = \min(M)\},\\
    P_4 &=& \{(n_0, n_1, n_2, n_3, n_4, n_5, n_6) \mid    n_0 = n_5, n_1 + n_2=n_3+n_4 = n_6 + n_0,\\
     & &n_0 = \min(M), n_0 < \min(\{n_1, n_4\}), n_0 < \max(\{n_2, n_3\})\}.
\end{eqnarray*}

\item Consider a position $M = (n_0,n_1,n_2,n_3,n_4, n_5,n_6,n_7)$ in ${\rm CN}(8,6)$. $M$ is a $\mathcal{P}$-position if and only if $M \in_\circlearrowleft P = \{(n_0,n_1,n_2,n_3,n_4, n_5,n_6,n_7)  \mid  n_0 = 0, n_1 = n_2 + n_3 = n_5+n_6 = n_7 , n_4 = \min(\{n_1, n_2+n_6\})\}$.
\end{itemize}

For the other cases of {\sc circuar nim}, it is still an open problem to find closed formulas for determining the set of $\mathcal{P}$-positions.

As shown here, in {\sc circular nim}, the formulas for $\mathcal{P}$-positions are completely different despite the rules are similar. This is the reason why almost all cases are still unsolved problems. 

In this paper, we aim to reexamine {\sc circular nim} from a broader perspective by extending its rules and explore the relationship between the rules and the formula that determines $\mathcal{P}$-positions.

\subsection{Sums of positions}
We also use the following two types of sums of positions.
\subsubsection{Disjunctive sum}
\begin{definition}
     Let $G$ and $H$ be positions of impartial rulesets. The {\em disjunctive sum} of $G$ and $H$, which is denoted as $G + H$, is a position such that the set of its all options is
    $$
    \{G' + H, G + H' \mid G' \text{ is an option of }G, H' \text{ is an option of }H\}.
    $$
\end{definition}
That is, in the disjunctive sum of positions, the current player chooses exactly one component and makes a legal move on the selected component. One can determine which player has a winning strategy by using Sprague-Grundy values.
\begin{definition}
Let $G$ be a position of an impartial ruleset. The Sprague-Grundy value of $G$ is denoted by $\mathcal{G}(G)$ and 

$$
\mathcal{G}(G) = {\rm mex}(\{\mathcal{G}(G') \mid G' \text{is an option of } G \}).
$$

Here, ${\rm mex}(S) = \mathbb{N}_0 \setminus S$ and $\mathbb{N}_0$ is the set of all nonnegative integers.

\end{definition}
\begin{theorem}[\cite{Spr35, Gru39}]
The following holds: 
\begin{itemize}
    \item Let $G$ be a position of an impartial ruleset. Then $G$ is a $\mathcal{P}$-position if and only if $\mathcal{G}(G) = 0$.
    \item Let $G$ and $H$ be positions of impartial rulesets. Then, $\mathcal{G}(G+H) = \mathcal{G}(G) \oplus \mathcal{G}(H).$
\end{itemize}    
\end{theorem}

For a position $(n_0, n_1, \ldots, n_{m-1})$ in {\sc nim} it is easy to confirm that its Sprague-Grundy value is $n_0 \oplus n_1 \oplus \cdots \oplus n_{m-1}$. On the other hand, for {\sc Moore's nim} ${\rm MN}(m, k)$, no closed formula for calculating the Sprague-Grundy value in the general case is known.

However, for ${\rm MN}(m, m - 1)$, a closed formula for calculating Sprague-Grundy values is shown  as follows.

\begin{theorem}[\cite{JM80}]
\label{thm:mooregrundy}
    Assume that $G = (n_0, n_1, \ldots, n_{m-1})$ is a position in ${\rm MN}(m, m-1)$ and $n_0 = \min(\{n_i \mid 0\le i \le m-1\})$. We also let $t_i$ be $i$-th triangular number, that is, $t_i = \frac{i(i+1)}{2}$. Then, the  Sprague-Grundy value of $G$ is 
    $$
    \left \{ \begin{array}{cc}
        t_x + ((n_0 - t_x - 1) \bmod (x + 1)) & (\text{If } t_x < n_0) \\
        n_0 + n_1 + \cdots + n_{m-1} & (\text{Otherwise}),
    \end{array} \right.
    $$
    where $x = (n_1 - n_0) + \cdots + (n_{m-1} - n_0)$.
\end{theorem}

\subsubsection{Selective sum}
\begin{definition}

    Let $G$ and $H$ be positions of impartial rulesets. The {\em selective sum} of $G$ and $H$ , which is denoted as $G \vee H$, is a position such that the set of its all options is
    $$
    \{G' \vee H, G \vee H', G' \vee H' \mid G' \text{ is an option of }G, H' \text{ is an option of }H\}.
    $$
    
\end{definition}
That is, in the selective sum of positions, the current player chooses at least one component and makes a legal move on each selected component.
\begin{theorem}[\cite{Smi66}]
\label{thm:selsum}
    Let $G$ and $H$ be positions of impartial rulesets. $G \vee H$ is a $\mathcal{P}$-position if and only if both $G$ and $H$ are $\mathcal{P}$-positions.
\end{theorem}

The rest of this paper is organized as follows: In Section \ref{sec_ecn}, we introduce the rule of {\sc extended circular nim} and show basic results. 
In Sections \ref{sec_ecn6}, \ref{sec_ecn7}, and \ref{sec_ecn8}, we present results on {\sc extended circular nim} with six, seven, and eight piles, respectively. In these sections, we also prove Theorems \ref{thm_2m2}, \ref{thm_ecn2m12m1}, and \ref{thm_2m2m2}, which can be used in general cases.

\section{{\sc Extended circular nim}}
\label{sec_ecn}

{\sc Extended circular nim} ${\rm ECN}(m_S, k)$ is defined as follows.

\begin{definition}
Let $m$ and $k$ be positive integers with $k \leq m$ and let $S$ be a set of positive integers that are not more than half of $m$.
{\sc Extended circular nim} ${\rm ECN}(m_S, k)$ is {\sc nim} on a simplicial complex $\Delta$, where $V = \{v_0, \ldots, v_{m-1}\}$ and $\Delta$ is the set of all non-empty subset of $\bigcup_{s \in S} \bigcup_{i = 0}^{m - 1}  \{v_{i}, v_{((i+s) \bmod m)}, \ldots, v_{((i + (k-1)s) \bmod m )}\}$. 
\end{definition}

That is, in {\sc extended circular nim} ${\rm ECN}(m_S, k)$, there are $m$ piles of tokens arranged in a circle. $S$ is a set of positive integers less than or equal to half of $m$.
Each player, in their turn, chooses at most $k$ piles selected every $s$-th pile in a circle for an $s \in S$.

Figure \ref{6123} shows piles of a position in ${\rm ECN}(6_{\{1,2\}},3)$. For example, selecting piles $B, C,$ and $D$ is legal because the player selects every first pile (connected black lines), also, selecting piles $A, C$, and $E$ is legal because the player selects every second pile (connected green lines). However, selecting piles $A, B$, and $D$ is illegal in this ruleset.
\begin{figure}[tb]
\centering
\includegraphics[width = 10cm]{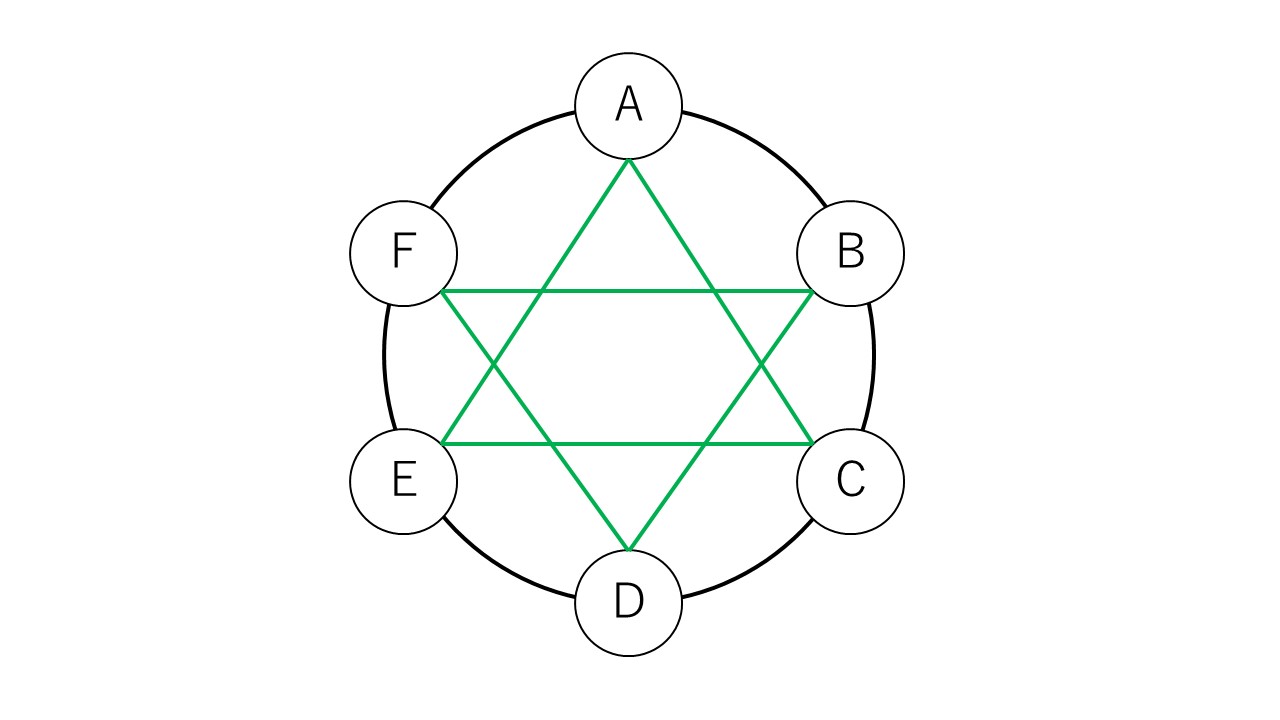}
\caption{${\rm ECN}(6_{\{1,2\}}, 3)$.}
\label{6123}
\end{figure}

\begin{definition}
    Let $\Gamma_1$ and $\Gamma_2$ be rulesets. $\Gamma_1$ and $\Gamma_2$ are {\em isomorphic} if there is a bijection map $f$ from the set of all positions in $\Gamma_1$ to the set of all positions in $\Gamma_2$ such that if $g_1, g_2$ are positions in $\Gamma_1$ and $g_2$ is an option of $g_1,$ then $f(g_2)$ is an option of $f(g_1)$ and vice versa.
\end{definition}
Note that from the definition of {\sc circular nim} and {\sc extended circular nim}, ${\rm ECN}(m_{\{1\}}, k)$ and ${\rm CN}(m, k)$ are isomorphic.

\begin{theorem}
\label{gcnas}
Let $a$ be an integer which satisfies $a \geq 4$.

When an integer $b \in S$ is prime to $a$, then ${\rm ECN}(a_S, a - 1)$ is isomorphic to ${\rm MN}(a, a-1)$. Otherwise, it is isomorphic to ${\rm ECN}(a_S, a-2)$.

When an integer $b \in S$ is prime to $a$, then ${\rm ECN}(a_S, a)$ is isomorphic to ${\rm MN}(a, a)$; that is, this can be regarded as a one-pile {\sc nim}. Otherwise, it is isomorphic to ${\rm ECN}(a_S, a-2)$.
\end{theorem}

\begin{proof}
Consider ${\rm ECN}(a_S, a - 1)$.
If $b \in S$ is prime to $a$, the player can choose arbitrary $a-1$ piles. Therefore, ${\rm ECN}(a_S, a - 1)$ is isomorphic to ${\rm MN}(a, a-1)$.
Otherwise, the player can choose at most $\frac{a}{2}$ piles. Since $\frac{a}{2} \leq a-2$, this ruleset is isomorphic to ${\rm ECN}(a_S, a-2)$.

Similarly, ${\rm ECN}(a_S, a)$ is isomorphic to ${\rm MN}(a, a)$ when an integer $b \in S$ is prime to $a$ and otherwise ${\rm ECN}(a_S, a)$ is isomorphic to ${\rm ECN}(a_S, a-2)$.
\end{proof}

Table \ref{matome1} summarizes these discussions.
From them, when considering {\sc extended circular nim}, we do not need to consider ${\rm ECN}(a_S, 1), {\rm ECN}(a_S, a - 1), {\rm ECN}(a_S, a)$ independently.
Thus, in this paper, we consider ${\rm ECN}(a_S, i)(1 < i < a - 1)$.

\begin{table}[H]
\begin{center}
\begin{tabular}{c|c}
Ruleset & Result \\ \hline
${\rm ECN}(a_S, 1)$ & Isomorphic to $a$-pile {\sc nim} \\
${\rm ECN}(a_S, a - 1)$ & Isomorphic to ${\rm MN}(a, a-1)$ or ${\rm ECN}(a_S, a-2)$ \\
${\rm ECN}(a_S, a)$ & Isomorphic to ${\rm MN}(a, a)$ or ${\rm ECN}(a_S, a-2)$ \\ \hline
\end{tabular}
\caption{Results for easy cases}
\label{matome1}
\end{center}
\end{table}

\subsection{{\sc Extended circular nim} with less than six piles}

We can characterize $\mathcal{P}$-positions in {\sc extended circular nim} with less than six piles by using early results.

First, we consider ${\rm ECN}(4_S, 2)$.
When $S = \{1\}$, the ruleset is already considered as a case of {\sc circular nim} in \cite{DH13}.
When $S = \{2\},$ we can regard each pair of diagonal piles as a single pile. 
Therefore, the position can be considered as a position in two-pile {\sc nim}. When $S = \{1, 2\},$ arbitrary two piles can be selected in ${\rm ECN}(4_{\{1,2\}}, 2)$. Therefore, this ruleset is isomorphic to ${\rm MN}(4,2).$ 

Next, we consider ${\rm ECN}(5_S, k)(k = 2, 3)$. 
When $S = \{1\}$, the ruleset is already considered as cases of {\sc circular nim} in \cite{ES96,Duf96}.
As shown in Fig. \ref{doukei}, ${\rm ECN}(5_{\{2\}}, i)$ is isomorphic to ${\rm ECN}(5_{\{1\}}, i)$. 
When $S = \{1, 2\}$, arbitrary two piles can be selected in ${\rm ECN}(5_{\{1, 2\}}, 2)$ and arbitrary three piles can be selected in ${\rm ECN}(5_{\{1, 2\}}, 3)$. Therefore, these rulesets are isomorphic to ${\rm MN}(5, 2)$ and ${\rm MN}(5, 3)$, respectively.

Table \ref{matome2} summarizes these results.
\begin{figure}[tb]
\centering
\includegraphics[width = 8cm]{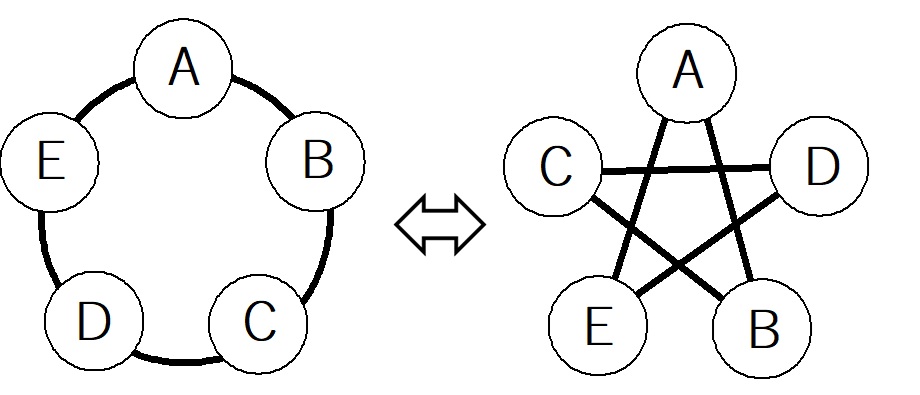}
\caption{${\rm ECN}(5_{\{1\}}, i)$ and ${\rm ECN}(5_{\{2\}},i)$ are isomorphic}
\label{doukei}
\end{figure}

\begin{table}[H]
\begin{center}
\begin{tabular}{c|c}
Ruleset & Result \\ \hline
${\rm ECN}(4_{\{1\}},2)$ & Shown in \cite{DH13} \\
${\rm ECN}(4_{\{2\}}, 2)$ & Regarded as two-pile {\sc nim} \\
${\rm ECN}(4_{\{1, 2\}}, 2)$ &  Isomorphic to ${\rm MN}(4,2)$ \\ \hline 
${\rm ECN}(5_{\{1\}},i)(2\leq i \leq 3)$ & Shown in \cite{ES96, Duf96}  \\
${\rm ECN}(5_{\{2\}},i)(2\leq i \leq 3)$ & Isomorphic to ${\rm ECN}(5_{\{1\}},i)$ \\
${\rm ECN}(5_{\{1,2\}},i)(2\leq i \leq 3)$ & Isomorphic to ${\rm MN}(5, i)$\\ \hline
\end{tabular}
\end{center}
\caption{Results for {\sc extended circular nim} with less than six piles}
\label{matome2}
\end{table}

\section{{\sc Extended circular nim} with six piles}
\label{sec_ecn6}
When there are six or more piles, some cases are not isomorphic to rulesets which are already studied.
We show the characterizations of $\mathcal{P}$-positions for such rulesets.

First, consider ${\rm ECN}(6_{\{2\}}, 2)$.
This ruleset can be considered as a disjunctive sum of two positions in ${\rm MN}(3,2)$. 
Therefore, by using Theorem \ref{thm:mooregrundy},  the set of $\mathcal{P}$-positions can be characterized.

In ${\rm ECN}(6_{\{2\}}, 3)$ and ${\rm ECN}(6_{\{2\}}, 4)$, The position $(n_0, n_1, n_2, n_3, n_4, n_5)$ can be considered as a position $(n_0 + n_2 + n_4, n_1 + n_3 + n_5)$ in two-pile {\sc nim}. Therefore, $(n_0, n_1, n_2, n_3, n_4, n_5)$ is a $\mathcal{P}$-position if and only if $n_0 + n_2 + n_4 = n_1 + n_3 + n_5.$ Also, in ${\rm ECN}(6_{\{3\}}, 2), {\rm ECN}(6_{\{3\}}, 3),$ and ${\rm ECN}(6_{\{3\}}, 4), $ the position  $(n_0, n_1, n_2, n_3, n_4, n_5)$ can be considered as a position $(n_0 + n_3, n_1 + n_4, n_2 + n_5)$ in three-pile {\sc nim}. Therefore, $(n_0, n_1, n_2, n_3, n_4, n_5)$ is a $\mathcal{P}$-position if and only if $(n_0 + n_3) \oplus (n_1 + n_4) \oplus (n_2 + n_5) = 0$.

Next, we consider ${\rm ECN}(6_{\{1,2\}}, 2)$.
\begin{theorem}
\label{gcn6122}
Consider a position $M = (n_0, n_1, n_2, n_3, n_4, n_5)$ in ${\rm ECN}(6_{\{1,2\}}, 2)$. $M$ is a $\mathcal{P}$-position if and only if $M \in P = \{(n_0, n_1, n_2, n_3, n_4, n_5) \mid n_0 \oplus n_3  = n_1 \oplus n_4 = n_2 \oplus n_5 \}$.
\end{theorem}
\begin{proof}
Let $N$ be the set of positions which are not in $P$.
It is clear that the terminal position is in $P$. 

Next, consider $M = (n_0, n_1, n_2, n_3, n_4, n_5) \in P$. Let $k = n_0 \oplus n_3  = n_1 \oplus n_4 = n_2 \oplus n_5$. Then, for any $M' = (n'_0, n'_1, n'_2, n'_3, n'_4, n'_5)$ which is an option of $M$, let $a_1 = n'_0 \oplus n'_3, a_2 = n'_1 \oplus n'_4, a_3 = n_2' \oplus n'_5$. At least one $a_i$, $a_i \neq k$ and there exists an $a_j$, $a_j = k$. Therefore, $M' \not \in P$.

Lastly, consider $M = (n_0, n_1, n_2, n_3, n_4, n_5)  \in N$. Let $k = \min(n_0 \oplus n_3, n_1 \oplus n_4, n_2 \oplus n_5)$. Without loss of generality, assume that $k = n_0 \oplus n_3$. Then, since $M \not \in P$, at least one of $n_1 \oplus n_4 > k$ and $n_2 \oplus n_5>k$ holds. In addition, each of $(n_1), (n_4), (n_2), (n_5), (n_1, n_2), (n_1, n_5), (n_4, n_2)$ and $(n_4, n_5)$ can be reduced in one move in ${\rm ECN}(6_{\{1,2\}}, 2)$. Therefore, $M$ has  an option $M'$ that satisfies $M' = (n_0, n'_1, n'_2, n_3, n'_4, n'_5)$, where $ n_0 \oplus n_3  = n'_1 \oplus n'_4 = n'_2 \oplus n'_5$.
\end{proof}

\begin{theorem}
\label{gcn6123}
Consider a position $M = (n_0, n_1, n_2, n_3, n_4, n_5)$ in ${\rm ECN}(6_{\{1,2\}}, 3)$. $M$ is a $\mathcal{P}$-position if and only if $M \in P = \{(n_0, n_1, n_2, n_3, n_4, n_5) \mid n_0 = n_3, n_1 = n_4, n_2 = n_5\}$.
\end{theorem}

\begin{proof}
Consider a position $(n_0, n_1, n_2, n_3, n_4, n_5)$ in ${\rm ECN}(6_{\{1,2\}}, 3)$. This position is a selective sum of $(n_0, n_3)$, $(n_1, n_4)$, and $(n_2, n_5)$, which are positions of two-pile {\sc nim}. Therefore, $n_0 \oplus n_3 = n_1 \oplus n_4 = n_2 \oplus n_5 = 0$, that is, $n_0 = n_3, n_1 = n_4, n_2 = n_5$ is a necessary and sufficient condition for that $(n_0, n_1, n_2, n_3, n_4, n_5)$ is a $\mathcal{P}$-position.
    
\end{proof}

\begin{theorem}
\label{gcn6124}
Let $P$ be the set of positions in ${\rm ECN}(6_{\{1,2\}}, 4)$ that satisfies the following: 
\begin{itemize}
\item Assume that $(n_0, n_1, n_2, n_3, n_4, n_5) \in P$. When each $n_i$ is expanded in the binary form, the summand of each digit contains four or no $1$s. In addition, for each digit contains four $1$s, the $1$s are in $(n_i, n_{(i+1) \bmod 6}, n_{(i+3) \bmod 6}, \allowbreak n_{(i+4) \bmod 6})$ or $(n_i, n_{(i+1) \bmod 6}, n_{(i+2) \bmod 6}, n_{(i+4) \bmod 6})$ for an integer $i \in \{0, 1, 2, 3, 4, 5\}$.  
\end{itemize}
Consider a position $M = (n_0, n_1, n_2, n_3, n_4, n_5)$ in ${\rm ECN}(6_{\{1,2\}}, 4)$.
Then, M is a $\mathcal{P}$-position if and only if $M \in P$.
\end{theorem}
\begin{proof}
We prove this theorem by using Theorem \ref{thm:nimregular}. Note that $\{\{v_{i}, v_{(i + 1)\bmod 6 }, v_{(i + 3)\bmod 6}, v_{(i + 4)\bmod 6}\} \mid 0 \le i \le 5\}= \{\{v_{i}, v_{(i + 1)}, v_{(i + 3)}, v_{(i + 4)\bmod 6}\} \mid 0 \le i \le 2\}.$
Let $\mathcal{B} = \{\emptyset\} \cup \{\{v_{i}, v_{(i+1) \bmod 6}, v_{(i + 2) \bmod 6}, v_{(i +4) \bmod 6}\} \mid 0 \le i \le 5\} \cup \{\{v_{i}, v_{(i + 1)}, v_{(i + 3)}, v_{(i + 4)\bmod 6}\} \mid 0 \le i \le 2\}$. We show $\mathcal{B}$ satisfies the three conditions of  {\sc nim}-basis.

For (A), it is clear that $\emptyset \in \mathcal{B}$.

For (B), assume that $F \in \Delta$ and $B, B' \in \mathcal{B}$ satisfy $B' = F \sqcup B$. 
When $B = \emptyset,$ since $B' = F$ and $F \not \in \mathcal{B}$ if $|F| > 0$, $B' = F = \emptyset$ holds. For the other cases, since each of $B$ and $B'$ has four vertices, $B = B'$ and $F$ must be the empty face.

For (C), assume that $F \in \Delta$ and $S$ is a subset of $V$. We show that there exist faces $K, G \in \Delta$, with $K \subseteq F \subseteq G,$ such that $G - F \subseteq S$ and $(S - G) \sqcup K \in \mathcal{B}$. 

If $S \cup F$ is a face of $\Delta$, then let $G = S \cup F, K = \emptyset$. We have $(G - F) \subseteq S$ and $(S - G) \sqcup K = \emptyset \in \mathcal{B}$.

If $S \cup F$ is not a face of $\Delta$, we consider three cases. Note that if a subset of $V$ has less than four vertices, it is a face of $\Delta.$ If $S \cup F$ is a face of $\Delta,$ we consider the following three cases.

\begin{itemize}
    \item[(i)] When $S \cup F$ has four vertices and is not a face of $\Delta$, we note that $S \cup F \in \mathcal{B}$. Let $G = K = F$. Then, $G - F = \emptyset \subseteq S$ and $(S - G)\sqcup K = S \cup F \in \mathcal{B}$.
    
    \item[(ii)] When $S \cup F$ has five vertices, without loss of generality, we assume that $S \cup F = \{v_0, v_1, v_2, v_3, v_4\}$. We consider five cases:
    \begin{itemize}
        \item[(ii-a)] When $F = \{v_0, v_4\}$, let $G = F \cup \{v_2\}$ and $K = F$. Then, $G - F = \{v_2\} \subset S$ and $(S - G) \sqcup K = \{v_0, v_1, v_3, v_4\} \in \mathcal{B}$.
        \item[(ii-b)] When $F (\neq \{v_0, v_4\})$ has two, three, or four vertices, let $G=  F$. Then, we have  $G-F = \emptyset \subseteq S$. Since $F$ has at least two vertices, $S - G = S - F$ has at most three vertices. 
        
        If $v_1 \in F,$
        $S - G$ is a subset of $\{v_0, v_2, v_3, v_4\}$, then $K = F - \{v_1\}$ satisfies $K \subset F$ and $(S - G) \sqcup K = S \cup F - \{v_1\} = \{v_0, v_2, v_3, v_4\} \in \mathcal{B}$. 
        
        Similarly, if $v_2 \in F$ (resp. $v_3 \in F),$ then $K = F - \{v_2\}$ (resp. $K = F - \{v_3\}$) satisfies $K \subset F$ and $(S - G) \sqcup K \in \mathcal{B}$.
       
        Since $F$ has at least two vertices, the remaining case is $\{v_0, v_4\} \subset F$. 
        From the rule of ${\rm ECN}(6_{\{1, 2\}}, 4)$ and $v_5 \not \in F$, $F = \{v_0, v_2, v_4\}$ or $F = \{v_0, v_4\}$ holds, but the latter case has been considered. For the former case, let $K = F - \{v_2\}$. We have $K \subset F$ and $(S - G) \sqcup K = S \cup F - \{v_2\} = \{v_0, v_1, v_3, v_4\} \in \mathcal{B}$.  
        \item[(ii-c)] When $F = \{v_1\}, F = \{v_2\}, $ or $F = \{v_3\}$, let $G = F$ and $K = \emptyset$. Then, $G - F = \emptyset \subseteq S$ and $(S - G) \sqcup K = S - G \in \{\{v_0, v_2, v_3, v_4\}, \{v_0, v_1, v_3, v_4\}, \{v_0, v_1, v_2, v_4\}\} \subset \mathcal{B}$.
        \item[(ii-d)] When $F = \{v_0\}$ (resp. $F = \{v_4\}$), let $G = \{v_0, v_1\}$ (resp. $G = \{v_3, v_4\}$) and $K = F$. Then, $G - F = \{v_1\} \subseteq S$ (resp. $G - F = \{v_3\} \subseteq S$) and $(S - G) \sqcup K = \{v_0, v_2, v_3, v_4\} \in \mathcal{B}$ (resp. $(S - G) \sqcup K = \{v_0, v_1, v_2, v_4\} \in \mathcal{B}$). 
        \item[(ii-e)] When $F = \emptyset$, let $G = \{v_2\}$ and $K = F = \emptyset$. Then, $G - F  = \{v_2\} \subseteq S$ and $(S -G) \sqcup K = \{v_0, v_1, v_3, v_4\} \in \mathcal{B}$. 
    \end{itemize}
    \item[(iii)] When $S \cup F$ has six vertices, without loss of generality, we assume that $F \in \{ \emptyset, \{v_0\}, \{v_0, v_1\}, \{v_0, v_2\}, \allowbreak \{v_0, v_3\}, \{v_0, v_1. v_2\}, \{v_0, v_1, v_3\}, \{v_0, v_2, v_4\}, \{v_0, v_1, v_2, v_3\}\}.$
    We consider four cases: 
    \begin{itemize}
        \item[(iii-a)] When $F = \emptyset, F = \{v_0\}$, or $F = \{v_0, v_2\}$, let $G = \{v_0, v_2\}$ and $K = \emptyset$. Then, $F \subseteq G$ and $(G - F) \subseteq S$. Also, $K \subseteq F$ and $(S - G) \sqcup K = \{v_1, v_3, v_4, v_5\} \in \mathcal{B}$ hold.
        \item[(iii-b)] When $F = \{v_0, v_1\},$ let $G = \{v_0, v_1, v_2\}$ and $K = \{v_1\}$. Then, $F \subseteq G$ and $(G - F) = \{v_2\} \subset S$. Also, $K \subseteq F$ and $(S - G) \sqcup K = \{v_1, v_3, v_4, v_5\} \in \mathcal{B}$ hold.
        \item[(iii-c)] When $F = \{v_0, v_3\}, F = \{v_0, v_1, v_3\},$ or $F = \{v_0, v_1, v_2, v_3\}$, let $G = F$ and $K = \{v_1, v_2, v_4, v_5\} \cap F \subseteq F$. Note that $S - G = S - F$ is a subset of $\{v_1, v_2, v_4, v_5\} \in \mathcal{B}$. 
        Then, $G - F = \emptyset \subset S$ and $(S  - G) \sqcup K = (S - F) \sqcup (\{v_1, v_2, v_4, v_5\} \cap F) = \{v_1, v_2, v_4, v_5\} \in \mathcal{B}$ hold. 
        \item[(iii-d)] When $F = \{v_0, v_1, v_2\}$ or $F = \{v_0, v_2, v_4\},$ let $G = F$ and $K = \{v_1, v_3, v_4, v_5\} \cap F \subseteq F$. Note that $S - G = S - F$ is a subset of $\{v_1, v_3, v_4, v_5\}.$ Then, $G - F = \emptyset \in S$ and $(S - G) \sqcup K = (S - F) \sqcup (\{v_1, v_3, v_4, v_5\} \cap F) = \{v_1,v_3,v_4,v_5\} \in \mathcal{B}.$ 
    \end{itemize}
\end{itemize}







\end{proof}

For ${\rm ECN}(6_{\{1,3\}}, 2),$ we show the following more generalized theorem.

\begin{theorem}
\label{thm_2m2}
Let $m (> 1)$ be an integer and assume that $h$ is maximal odd integer not more than $m$.
Consider a position $M = (n_0, n_1, \ldots, n_{2m-1})$ in ${\rm ECN}((2m)_{\{1, 3, \ldots, h\}},2)$. Then, $M$ is a $\mathcal{P}$-position if and only if $M \in P=\{(n_0, n_1, \ldots, n_{2m-1}) \mid n_0 \oplus n_2 \oplus \cdots \oplus n_{2m-2} = n_1 \oplus n_3 \oplus \cdots \oplus n_{2m-1} = 0\}$.
\end{theorem}
\begin{proof}
        In ${\rm ECN}((2m)_{\{1,3,\ldots, h\}}, 2)$, the player can reduce at most one of $n_0, n_2, \ldots, n_{2m-2}$ and at most one of $n_1, n_3, \ldots, n_{2m-1}$ in one move. Therefore, this ruleset is the selective sum of two positions $(n_0, n_2, \ldots, n_{2m-2})$ and $(n_1, n_3, \ldots, n_{2m-1})$ in $m$-pile {\sc nim}. Thus, from Theorems \ref{thm_nim} and \ref{thm:selsum}, $M$ is a $\mathcal{P}$-position if and only if $n_0 \oplus n_2 \oplus \cdots \oplus n_{2m-2} = n_1 \oplus n_3 \oplus \cdots \oplus n_{2m-1} = 0.$
\end{proof}

\begin{corollary}
\label{gcn6132}
Consider a position $M = (n_0, n_1, n_2, n_3, n_4, n_5)$ in ${\rm ECN}(6_{\{1,3\}}, 2)$. Then, $M$ is a $\mathcal{P}$-position if and only if $M \in P = \{(n_0, n_1, n_2, n_3, n_4, n_5) \mid n_0 \oplus n_2 \oplus n_4 = n_1 \oplus n_3 \oplus n_5 = 0\}$.
\end{corollary}

Note that Theorem \ref{thm_2m2} is also a generalization of the result on ${\rm ECN}(4_{\{1\}}, 2)$.


\begin{theorem}
\label{gcn6233}
Consider a position $M = (n_0, n_1, n_2, n_3, n_4, n_5)$ in ${\rm ECN}(6_{\{2,3\}}, 3)$. $M$ is a $\mathcal{P}$-position if and only if $M \in_\circlearrowleft P = \{(n_0, n_1, n_2, n_3, n_4, n_5) \mid n_0 + n_2 + n_4 = n_1 + n_3 + n_5,  n_0 \oplus n_1 \oplus n_2 = 0,  n_0 \leq n_3, n_1 \leq n_4, n_2 \leq n_5\}$.
\end{theorem}
\begin{proof}
Let $N = \{M = (n_0, n_1, n_2, n_3, n_4, n_5) \mid M \not \in_\circlearrowleft P\}$. The terminal position is in $P$.

Next, consider a position $M \in_\circlearrowleft P$. Without loss of generality, we assume that $M = (n_0, n_1, n_2, n_3, n_4, n_5) \in P$.  Then, $n_0 \leq n_3, n_1 \leq n_4,$ and $ n_2 \leq n_5$ hold. Let $M' = (n'_0, n'_1, n'_2, n'_3, n'_4, n'_5)$ be an option of $M$. When the player reduces some of $n_0, n_2, n_4$, or reduces some of $n_1, n_3, n_5$, $n'_0 + n'_2 + n'_4 \neq n'_1 + n'_3 + n'_5$ holds and $M' \not \in_\circlearrowleft P$.  For the case that the player reduces the same number from $n_0$ and $n_3$, $n'_0 + n'_2 + n'_4 = n'_1 + n'_3 + n'_5$ is satisfied but $n'_0 \oplus n'_1 \oplus n'_2  = n'_0 \oplus n_1 \oplus n_2 \neq 0$ holds. We show $M' \not \in_\circlearrowleft P$ by contradiction: 

\begin{itemize}
\item [(i)]Assume that $n'_1 \oplus n'_2 \oplus n'_3 = 0$. If $n_0 = n_3$, then $n'_0 = n'_3$, which contradicts to $n'_0 \oplus n'_1 \oplus n'_2 \neq 0$. In addition, $n_0 < n_3$ does not satisfy $n'_3 \leq n'_0$.

\item [(ii)]Assume that $n'_2 \oplus n'_3 \oplus n'_4 =  n_2 \oplus n'_3 \oplus n_4 = 0$. In order to satisfy  $M' \in_\circlearrowleft P$, it is necessary to be $n'_3 \leq n'_0, n_4 \leq n_1$. From the assumption, $n'_0 \leq n'_3, n_1 \leq n_4$, so $n'_0 = n'_3, n_1 = n_4$ hold. However, then we have $n_2 \oplus n'_3 \oplus n_4 = n_2 \oplus n'_0 \oplus n_1 = 0$, which contradicts to $n'_0 \oplus n_1 \oplus n_2 \neq 0$.

\item [(iii)]Assume that $n'_3 \oplus n'_4 \oplus n'_5 = n'_3 \oplus n_4 \oplus n_5 = 0$. In order to satisfy $M' \in_\circlearrowleft P$, it is necessary to be $n'_3 \leq n'_0, n_4 \leq n_1, n_5 \leq n_2$. From the assumption, $n'_0 \leq n'_3, n_1 \leq n_4, n_2 \leq n_5$, so $n'_0 = n'_3, n_1 = n_4, n_2 = n_5$ hold. This contradicts to $n'_0 \oplus n_1 \oplus n_2 \neq 0$.
\end{itemize}

Even if we assume that $n'_4 \oplus n'_5 \oplus n'_0 = 0$ or $n'_5 \oplus n'_0 \oplus n'_1 = 0$, we have a contradiction in a similar way. Therefore, $M' \in N$. 

For the cases that the same number is reduced from $n_1$ and $n_4$, or the same number is reduced from $n_2$ and $n_5$, we have $M' \in N$ in similar ways. 


Finally, consider a position $M = (n_0, n_1, n_2, n_3, n_4, n_5) \in N$.

\begin{itemize}
\item[(i)]Assume that $n_0 + n_2 + n_4 \neq n_1 + n_3 + n_5 $. Without loss of generality, $n_0 + n_2 + n_4 > n_1 + n_3 + n_5$.
Then, $n_0 > n_3$, $n_2 > n_5$, or $n_4 > n_1$ holds. Without loss of generality, we assume that $n_0 > n_3$.
\begin{itemize}
\item[(i-a)]Assume that $n_2 \leq n_5, n_4 \leq n_1$.
\begin{itemize}
\item[(i-a-1)]If $n_2 \oplus n_4 \leq n_3$, then let $n'_3 = n_2 \oplus n_4, n'_0 = n_1 + n'_3 + n_5 - n_2 - n_4$ and a move changes $(n_0, n_3)$ to $(n'_0, n'_3)$ is a legal move, which satisfies $n'_0 + n_2 + n_4 = n_1 + n'_3 + n_5, n_2 \oplus n'_3 \oplus n_4 = 0, n_2 \leq n_5, n'_3 \leq n'_0, n_4 \leq n_1$.
\item[(i-a-2)] If $n_2 \oplus n_4 > n_3$, then $n_2 \geq n_3 \oplus n_4$ or $n_4 \geq n_2 \oplus n_3$ holds. 
Without loss of generality, we assume that $n_2 \geq n_3 \oplus n_4$. Let $n'_2 = n_3 \oplus n_4 \leq n_2$. If $n_0 + n'_2 + n_4 < n_1 + n_3 + n_5$, then let $n'_5 = n_0 + n'_2 + n_4 - n_1 - n_3$. It satisfies $n'_2 \leq n'_5 < n_5$ and $n_0 + n'_2 + n_4 = n_1 + n_3 + n'_5, n'_2 \oplus n_3 \oplus n_4 = 0, n'_2 \leq n'_5, n_3 < n_0, n_4 \leq n_1$ hold. Thus, $M' = (n_0, n_1, n'_2, n_3, n_4, n'_5)$ is an option of $M$ and $M' \in_\circlearrowleft P$.  On the other hand, if $n_0 + n'_2 + n_4 \geq n_1 + n_3 + n_5$, then $n'_2 \leq n_5, n_4 \leq~n_1$, so by letting $n'_0 = n_1 + n_3 + n_5 - n'_2 - n_4$, we have  $n'_0 + n'_2 + n_4 = n_1 + n_3 + n_5, n'_2 \oplus n_3 \oplus n_4 = 0, \allowbreak n'_2 \leq n_5, n_3 \leq n'_0, n_4 \leq n_1$. Thus, $M' = (n'_0, n_1, n'_2, n_3, n_4, n_5)$ is an option of $M$ and $M' \in_\circlearrowleft P$.
\end{itemize}

\item[(i-b)]Assume that $n_2 > n_5, n_4 \leq n_1$. Note that at least one of $n_4 \oplus n_5 \leq n_3, n_4 \oplus n_3 \leq n_5$ and $ n_5 \oplus n_3 \leq n_4$ holds. We separate into five cases.
\begin{itemize}
\item[(i-b-1)]Consider the case that $n_4 \oplus n_5 \leq n_3$ and $n_2 + n_4 \leq n_1 + n_5$. Let $n'_3 = n_4 \oplus n_5, n'_0 = n_1 + n_5 - n_2 - n_4 + n'_3$, then $n'_0 + n_2 + n_4 = n_1 + n'_3 + n_5, n'_3 \oplus n_4 \oplus n_5 = 0, n'_3 \leq n'_0, n_5 < n_2, n_4 \leq n_1$ and $M' = (n'_0, n_1, n_2, n'_3, n_4, n_5) \in_\circlearrowleft P$. In addition, we have $n_3 \geq n'_3$ and $n_0 > n_1 + n_3 + n_5 - n_2 - n_4 \geq n_1 + n_5  + n'_3 - n_2 -n_4 = n'_0$ since $n_0 + n_2 + n_4 > n_1 + n_3 + n_5$. Thus, $M'$ is an option of $M$.
\item[(i-b-2)]
Consider the case that $n_4 \oplus n_5 \leq n_3$ and $n_2 + n_4 > n_1 + n_5$. Then, if 
$n_1 + n_3 + n_5 > (n_4 \oplus n_5) + n_2 + n_4$, that is, 
$n_3 >  n_2 + n_4 - n_1 - n_5 + (n_4 \oplus n_5) > (n_4 \oplus n_5)$, let 
$n'_3 = n_2 + n_4 - n_1 - n_5 + (n_4 \oplus n_5), n'_0 = n_4 \oplus n_5$.  We have $n'_0 + n_2 + n_4 = n_1 + n'_3 + n_5, n_4 \oplus n_5 \oplus n'_0 = 0, n_4 \leq n_1, n_5 < n_2, n'_0 < n'_3$. Thus, $M' = (n'_0, n_1, n_2, n'_3, n_4, n_5)$ is an option of $M$ and  $M' \in_\circlearrowleft P$.
On the other hand, if $n_1 + n_3 + n_5 \leq (n_4 \oplus n_5) + n_2 + n_4$, then since $n_3 \geq n_4 \oplus n_5$ and $n_1 \geq n_4$, there exists $n'_2$ which satisfies $n_5 \leq n'_2 \leq n_2$ and $n_1 + n_3 + n_5 = (n_4 \oplus n_5) + n'_2 + n_4$. Therefore,  let $n'_0 = n_4 \oplus n_5$ and we have $n'_0 + n'_2 + n_4 = n_1 + n_3 + n_5, n_4 \oplus n_5 \oplus n'_0 = 0, n_4 \leq n_1, n_5 \leq n'_2, n'_0 \leq n_3$. Thus, $M' = (n'_0, n_1, n'_2, n_3, n_4, n_5)$ is an option of $M$ and $M' \in_\circlearrowleft P$.
\item[(i-b-3)]
Consider the case that $n_4 \oplus n_3 \leq n_5$ and $n_0 + n_4 \leq n_1 + n_3$. Let $n'_5 = n_4 \oplus n_3, n'_2 = n_1 + n_3 - n_0 - n_4 + n'_5$, then $n_0 + n'_2 + n_4 = n_1 + n_3 + n'_5, n_3 \oplus n_4 \oplus n'_5 = 0, n_3 < n_0, n'_5 \leq n'_2, n_4 \leq n_1$ and $M' = (n_0, n_1, n'_2, n_3, n_4, n'_5) \in_\circlearrowleft P$. In addition, we have $n_5 \geq n'_5$ and $n_2 > n_1 + n_3 + n_5 - n_0 - n_4 \geq n_1 + n'_5  + n_3 - n_0 -n_4 = n'_2$
since $n_0 + n_2 + n_4 > n_1 + n_3 + n_5$. Thus, $M'$ is an option of $M$.
\item[(i-b-4)]Consider the case that $n_4 \oplus n_3 \leq n_5$ and $n_0 + n_4 > n_1 + n_3$. Then, if 
$n_1 + n_3 + n_5 > (n_4 \oplus n_3) + n_0 + n_4$, that is, 
$n_5 > n_0 + n_4 - n_1 - n_3 + (n_4 \oplus n_3) > (n_4 \oplus n_3)$, let 
$n'_5 = n_0 + n_4 - n_1 - n_3 + (n_4 \oplus n_3), n'_2 = n_4 \oplus n_3$.  We have $n_0 + n'_2 + n_4 = n_1 + n_3 + n'_5, n'_2 \oplus n_3 \oplus n_4 = 0, n'_2 < n'_5, n_3 < n_0, n_4 \leq n_1$. Thus, $M' = (n_0, n_1, n'_2, n_3, n_4, n'_5)$ is an option of $M$ and $M' \in_\circlearrowleft P$.
On the other hand, if $n_1 + n_3 + n_5 \leq (n_4 \oplus n_3) + n_0 + n_4$, then since $n_5 \geq n_4 \oplus n_3$ and $n_1 \geq n_4$, there exists $n'_0$ which satisfies $n_3 \leq n'_0 \leq n_0$ and $n_1 + n_3 + n_5 = (n_4 \oplus n_3) + n'_0 + n_4$. Therefore,  let $n'_2 = n_4 \oplus n_3$ and we have $n'_0 + n'_2 + n_4 = n_1 + n_3 + n_5, n'_2 \oplus n_3 \oplus n_4 = 0, n'_2 \leq n_5, n_3 \leq n'_0, n_4 \leq n_1$. Thus, $M' = (n'_0, n_1, n'_2, n_3, n_4, n_5)$ is an option of $M$ and $M' \in_\circlearrowleft P$.

\item[(i-b-5)]
Consider the case that $n_5 \oplus n_3 \leq n_4$.
Let $n'_4 = n_5 \oplus n_3$.
If $n'_4 + n_0 + n_2 \leq n_1 + n_3 + n_5$, for $n'_1 = n'_4 + n_0 + n_2 - n_3 - n_5$, we have $n'_1 > n'_4$ since $n_0 + n_2 - n_3 - n_5 > 0$. In addition, we have $n_3 \oplus n'_4 \oplus n_5 = 0, n_0 + n_2 + n'_4 = n'_1 + n_3 + n_5, n_3 < n_0, n_5 < n_2$. Thus, $M' = (n_0, n'_1, n_2, n_3,n'_4, n_5)$ is an option of $M$ and $M' \in_\circlearrowleft P$.
On the other hand, if $n'_4 + n_0 + n_2 > n_1 + n_3 + n_5$, since $n'_4 \leq n_1$, there exist $n'_0, n'_2$ which satisfy $n_5 \leq n'_2 \leq n_2, n_3 \leq n'_0 \leq n_0$ and $n'_0 + n'_2 + n'_4 = n_1 + n_3 + n_5$. Here, $n_3 \oplus n'_4 \oplus n_5 = 0, n'_4 \leq n_1$ also holds. Thus, $M' = (n'_0, n_1, n'_2, n_3,n'_4, n_5)$ is an option of $M$ and $M' \in_\circlearrowleft P$.
\end{itemize}

\item[(i-c)] For the case that $n_2 \leq n_5, n_4 > n_1$, we can show $M$ has an option $M' \in_\circlearrowleft P$ in a similar way to (i-b). 
\item[(i-d)]Assume that $n_2 > n_5, n_4 > n_1$. Then $n_1 \oplus n_3 \leq n_5, n_3 \oplus n_5 \leq n_1, $ or $ n_5 \oplus n_1 \leq n_3$ holds. Without loss of generality, $n_1 \oplus n_3 \leq n_5$. Let $n'_2 = n_1 \oplus n_3 < n_2$. 
If $n_0 + n'_2 + n_4 > n_1 + n_3 + n_5$, since $n'_2 \leq n_5$, there exist $n'_0, n'_4$ which satisfy $n_3 \leq n'_0 \leq n_0, n_1 \leq n'_4 \leq n_4$ and $n'_0 + n'_2 + n'_4 = n_1 + n_3 + n_5 $. Here, $n_1 \oplus n'_2 \oplus n_3 = 0$ holds. Thus, $M' = (n'_0, n_1, n'_2, n_3,n'_4, n_5)$ is an option of $M$ and $M' \in_\circlearrowleft 
P$.
On the other hand, if $n_0 + n'_2 + n_4 \leq n_1 + n_3 + n_5$, since $n_0 + n_4 > n_1 + n_3$, there exists $n'_5$ which satisfies $n'_2 < n'_5 \leq n_5 $ and $n_0 + n'_2 + n_4 = n_1 + n_3 + n'_5$. Here, $n_1 \oplus n'_2 \oplus n_3 = 0,$ $n_3 < n_0$, and $n_1 < n_4$. Thus, $M' = (n_0, n_1, n'_2, n_3,n_4, n'_5)$ is an option of $M$ and $M' \in_\circlearrowleft P$.
\end{itemize}
\item[(ii)]Assume that $n_0 + n_2 + n_4 = n_1 + n_3 + n_5$. Then 
without loss of generality, we assume that $n_0 \leq n_3, n_2 \leq n_5, n_4 \geq n_1$. Since $M \in N$, we have $n_0 \oplus n_1 \oplus n_2 \neq 0$. Here, consider a position $(n_0, n_1, n_2)$ in three-pile {\sc nim}. From this position, there is a move to make a $\mathcal{P}$-position. Assume that for $n'_0 < n_0$, $(n'_0, n_1, n_2)$ is a $\mathcal{P}$-position in three-pile {\sc nim}. Let $k =n_0 - n'_0$. Then, $M$ has an option $M' =  (n'_0, n_1, n_2, n_3 - k, n_4, n_5)$, which satisfies $n'_0 + n_2 + n_4 = n_1 + n_3 - k + n_5, n'_0 \oplus n_1 \oplus n_2 =0, n'_0 \leq n_3 - k, n_1 \leq n_4, n_2 \leq n_5$. Therefore, $M' \in_\circlearrowleft P$.
\end{itemize}
\end{proof}

Finally, it is easy to confirm that ${\rm ECN}(6_{\{1,3\}}, 4)$ is isomorphic to ${\rm ECN}(6_{\{1\}}, 4), {\rm ECN}(6_{\{2,3\}}, 4)$ is isomorphic to ${\rm ECN}(6_{\{2,3\}}, 3),{\rm ECN}(6_{\{1,2,3\}}, 2)$ is isomorphic to ${\rm MN}(6, 2),$ and ${\rm ECN}(6_{\{1,2,3\}}, 4)$ is isomorphic to ${\rm ECN}(6_{\{1,2\}}, 4).$

From these discussions, we have characterized $\mathcal{P}$-positions in six-pile {\sc extended circular nim} except for ${\rm ECN}(6_{\{1\}}, 2), {\rm ECN}(6_{\{1,3\}}, 3), {\rm ECN}(6_{\{2,3\}}, 2), $ and $ {\rm ECN}(6_{\{1,2,3\}}, 3)$. Table \ref{matome3} summarizes these results.

\begin{table}[H]
\begin{center}
\begin{tabular}{c|c}
Ruleset & Result \\ \hline
${\rm ECN}(6_{\{1\}}, 2)$ & Unsolved \\
${\rm ECN}(6_{\{1\}}, 3)$ & Shown in \cite{Hor10, DH13} \\
${\rm ECN}(6_{\{1\}}, 4)$ & Shown in \cite{DH13} \\
${\rm ECN}(6_{\{2\}}, 2)$ & Disjunctive sum of ${\rm MN}(3,2)$ \\
${\rm ECN}(6_{\{2\}}, i)(3 \leq i \leq 4)$ & Regarded as two-pile {\sc nim} \\
${\rm ECN}(6_{\{3\}}, i)(2 \leq i \leq 4)$ & Regarded as three-pile {\sc nim} \\
${\rm ECN}(6_{\{1,2\}}, 2)$ & Theorem \ref{gcn6122} \\
${\rm ECN}(6_{\{1,2\}}, 3)$ & Theorem \ref{gcn6123} \\
${\rm ECN}(6_{\{1,2\}}, 4)$ & Theorem \ref{gcn6124} \\
${\rm ECN}(6_{\{1,3\}}, 2)$ & Corollary \ref{gcn6132} \\
${\rm ECN}(6_{\{1,3\}}, 3)$ & Unsolved \\
${\rm ECN}(6_{\{1,3\}}, 4)$ & Isomorphic to ${\rm ECN}(6_{\{1\}}, 4)$ \\
${\rm ECN}(6_{\{2,3\}}, 2)$ & Unsolved \\
${\rm ECN}(6_{\{2,3\}}, 3)$ & Theorem \ref{gcn6233} \\
${\rm ECN}(6_{\{2,3\}}, 4)$ & Isomorphic to ${\rm ECN}(6_{\{2,3\}}, 3)$ \\
${\rm ECN}(6_{\{1,2,3\}}, 2)$ & Isomorphic to ${\rm MN}(6,2)$ \\
${\rm ECN}(6_{\{1,2,3\}}, 3)$ & Unsolved \\
${\rm ECN}(6_{\{1,2,3\}}, 4)$ & Isomorphic to ${\rm ECN}(6_{\{1,2\}},4)$ \\ \hline
\end{tabular}
\caption{Results of {\sc extended circular nim} with six piles}
\label{matome3}
\end{center}
\end{table}
\section{{\sc Extended circular nim} with seven piles}
\label{sec_ecn7}
In this section, we consider {\sc extended circular nim} with seven piles.

As mentioned above, {\sc circular nim} with seven piles, or {\sc extended circular nim} ${\rm ECN}(7_{\{1\}}, i)$ is solved when $i \in  \{1, 4, 6, 7\}$.

\begin{observation}
{\sc Extended circular nim} ${\rm ECN}(7_{\{2\}}, i)$ and ${\rm ECN}(7_{\{3\}}, i) (1\leq i \leq 7)$ are isomorphic to ${\rm ECN}(7_{\{1\}},i)$.

Also, ${\rm ECN}(7_{\{2,3\}}, i)$ and ${\rm ECN}(7_{\{1,3\}}, i)(1\leq i \leq 7)$ are isomorphic to ${\rm ECN}(7_{\{1,2\}}, i)$. 
\end{observation}
\begin{proof}
By properly exchanging the name of the piles, we can easily show the isomorphism of rulesets.

Figure \ref{7doukei} shows the case of  ${\rm ECN}(7_{\{1,2\}}, i)$ and  ${\rm ECN}(7_{\{1,3\}},i)$.
\end{proof}

\begin{figure}[tb]
\centering
\includegraphics[width = 8cm]{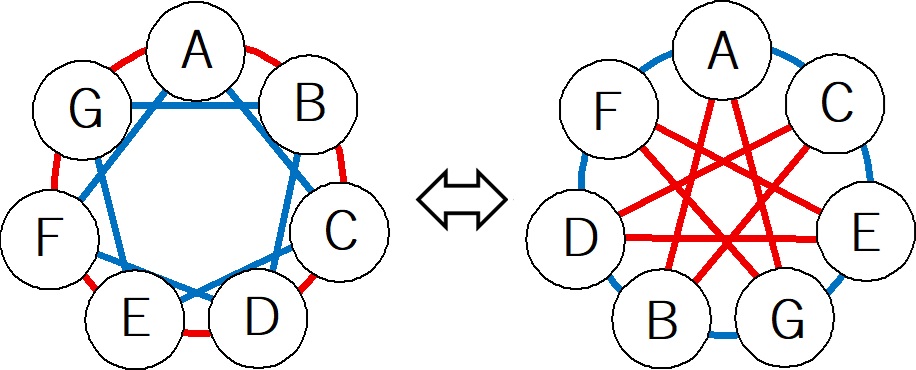}
\caption{${\rm ECN}(7_{\{1,2\}}, i)$ and ${\rm ECN}(7_{\{1,3\}},i)$ are isomorphic}
\label{7doukei}
\end{figure}

Therefore, without loss of generality, we only consider ${\rm ECN}(7_{\{1,2\}}, i), $ and $ {\rm ECN}(7_{\{1,2,3\}}, i)(2\leq i \leq 5)$.

\begin{theorem}
\label{gcn7124}
Consider a position $M = (n_0, n_1, n_2, n_3, n_4, n_5, n_6)$ in ${\rm ECN}(7_{\{1,2\}}, 4)$.
Then, $M$ is a $\mathcal{P}$-position if and only if $M \in_\circlearrowleft P = \{ (n_0, n_1, n_2, n_3, n_4, n_5, n_6) \mid  n_0 = n_1 = n_4 = \min(\{n_i\}), n_2 = n_6, n_3 + n_5 = n_0 + n_2\}$. 

\end{theorem}
\begin{proof}
Let 
$N = \{M = (n_0, n_1, n_2, n_3, n_4, n_5, n_6) \mid M \not \in_\circlearrowleft P\}$.
The terminal position is in $P$.


Next, consider a position $M \in_\circlearrowleft P$. Without loss of generality, we assume that $M = (n_0, n_1, n_2, n_3, n_4, n_5, n_6) = (a, a, a+b, a+c, a, a+d, a+b) \in P,$ where $a, b, c, d$ are nonnegative integers and $c + d = b$. 
Let $M' = (n'_0, n'_1, n'_2, n'_3, n'_4, n'_5, n'_6)$ be an option of $M$. Assume that $\min(\{n'_i\}) = n_0 = a$. Then, $n'_0 = n'_1 = n'_4 = a$ holds.
Note that $\{v_2, v_3, v_5, v_6\}, \{v_2, v_3, v_6\},$ and  $\{v_2, v_5, v_6\}$ are not faces, but every other subset of $\{v_2, v_3, v_5, v_6\}$ is a face in this ruleset. There are three cases.

\begin{itemize}
\item[(i)]
When $M' = (a, a, a+b', a+c', a, a+d', a+b) (b' + c' + d' < b + c + d, b' \leq b, c' \leq c, d' \leq d)$, we can say that $M' \not \in_\circlearrowleft P$, because if $M' \in_\circlearrowleft P,$ one of $a+b' = a+b, a+c' = a+b, $ and $a + b' = a+d'$ holds, but it cannot happen by the following discussion: 

\begin{itemize}
    \item [(i-a)] When $b = b', c' + d' < b$ and we have $M \not \in_\circlearrowleft P$.
    \item [(i-b)] When $c' = b,$ since $c + d = b,$ we have $c' = c  =b$ and $d = d' = 0$. Since $b'+c'+d' < b+c+d,$ we also have $b'<b$. Then, $M' = (a, a, a+b', a+b, a, a, a+b)$ and we consider if a rotated position $(a, a, a+b, a,a, a+b', a+b) \in P$. However, since $0+b' <  b, (a, a, a+b, a,a, a+b', a+b) \not \in P$ holds. Therefore, $M' \not \in_\circlearrowleft P$. 
    \item [(i-c)] When $b = b' = d',$ since $c + d = b,$ we have $d = d' = b$ and $c = c' = 0$. However, it does not satisfy $b' + c' + d' < b+ c+ d$.
    \item [(i-d)] When $b > b' = d', M = (a, a, a+b', a+c', a, a+b', a+b).$ If a rotated position $(a+c', a, a+b', a+b, a,a, a+b') \in P$, then
    $b + 0 = b'$ holds, which is a contradiction. Also, if another rotated position $(a, a+b', a+b, a, a, a+b', a+c') \in P$, then $b' = 0, b = c',$ and $0 + 0 = b$ holds, which is a contradiction. It is easy to confirm that every other rotated and flipped positions are not elements of $P$.  
\end{itemize}

\item[(ii)]
For $M' = (a, a, a+b, a+c', a, a+d', a+b')(b'+c'+d' < b+c+d, b' \leq b, c' \leq c, d' \leq d)$, we can say that $M' \not \in_\circlearrowleft P$ in a similar way.

\item[(iii)]
for $M' = (a, a, a+b', a+c,a, a+d, a+b'')(b'   + b'' < 2b , b', b'' \leq b ), $ we can say that $M' \not \in_\circlearrowleft P$ because if $b' = b''$ then $c + d = b> b',$  if $b' = d$ then $c > 0$ or $b'' < b$, and if $b'' =c$ then $d > 0$ or $b' < b$.

\end{itemize}

Assume that $\min(\{n'_i\}) < n_0$. If $M' \in_\circlearrowleft P$, then there exists $i$ such that $n'_i = n'_{((i+1) \bmod 7)} = n'_{((i+4) \bmod 7)} < n_0$, but selecting these three piles in one move is not legal in ${\rm ECN}(7_{\{1,2\}},4)$. Therefore, $M' \in N$.


Finally, consider a position $M = (n_0, n_1, n_2, n_3, n_4, n_5, n_6) \in N$. Without loss of generality, we assume $n_0  = \min(\{n_i\})$. Then $M = (a, a + a_1, a+a_2, a+a_3, a+a_4, a+a_5 , a+a_6)$, where $a_i \ge 0$ for every $i$.

Without loss of generality, we consider the case of $a_2 \geq a_5$.
\begin{itemize}
\item[(i)]

Assume that $a_2 \geq a_5, a_6 > a_5$.
\begin{itemize}
\item[(i-a)]
If $a_2 \geq a_6, a_3 + a_5 \geq a_6$, then $M' = (a, a, a+ a_6, a + a_6 - a_5, a, a + a_5, a  + a_6)$ is an option of $M$ and $M' \in P$.
\item[(i-b)]
If $a_2 \geq a_6, a_3 + a_5 < a_6$, then $M' = (a, a, a+a_3 +a_5, a+a_3, a, a+a_5, a+a_3 + a_5)$ is an option of $M$ and $M' \in P$. 
\item[(i-c)]
If $a_2 < a_6, a_3 + a_5 \geq a_2$, then $M'=(a, a, a+a_2, a+a_2 - a_5, a, a+a_5, a+a_2)$ is an option of $M$ and $M' \in P$.
\item[(i-d)]
If $a_2 < a_6, a_3 + a_5 < a_2$, then $(a, a, a+a_3 + a_5, a+a_3, a, a+a_5, a+a_3  +a_5)$ is an option of $M$ and $M' \in P$.
\end{itemize}
\item[(ii)]
If $a_2 \geq a_5, a_6 \leq a_5, a_1 + a_6 \geq a_5$, then $M'=(a, a+a_5 - a_6, a+a_5, a, a, a+a_5, a+a_6)$ is an option of $M$ and $M' \in_\circlearrowleft P$.
\item[(iii)]
If $a_2 \geq a_5, a_6 \leq a_5, a_1 + a_6 < a_5$, then $M'= (a, a+a_1, a+a_1 + a_6, a, a, a+a_1 + a_6, a+a_6)$ is an option of $M$ and $M' \in_\circlearrowleft P$.

\end{itemize}

Therefore, $M \in N$ has an option $M' \in_\circlearrowleft P$.

\end{proof}

For ${\rm ECN}(7_{\{1,2\}}, 5)$, we show the following more generalized theorem.

\begin{theorem}
\label{thm_ecn2m12m1}
Let $m (>1)$ be an integer and assume that $2m+1$ is a prime number.
Consider a position $M = (n_0, n_1, \ldots, n_{2m})$ in ${\rm ECN}((2m+1)_{\{1, 2, \ldots, m - 1\}}, 2m-1)$. $M$ is a $\mathcal{P}$-position if and only if $M \in_\circlearrowleft P = \{M \mid n_0 = 0, n_1 = n_2 = \cdots = n_{m-1} = n_{m}+n_{m+1} = n_{m+2} = \cdots = n_{2m}\}$.
\end{theorem}

\begin{proof}

Let $N = \{M=(n_0, n_1, \ldots, n_{2m})\mid M \not \in_\circlearrowleft P\}$. The terminal position is in $P$.

Next, consider a position $M \in_\circlearrowleft P$. Without loss of generality, we assume that $M = (n_0, n_1, \ldots, n_{2m}) \in P$.
From the definition of $P$, for any position $M \in P$, $(2m-1) \times \max(\{n_i\}) = \sum{n_i}$ holds.

Let $M' = (n'_0, n'_1, \ldots, n'_{2m})$ be an option of $M$. If $n'_i = n_i$ for $i \in \{1, 2, \ldots, m-1\} \cup \{m+2, \ldots, 2m\}$, then $\max(\{n_i\}) = \max(\{n'_i\})$, but $\sum{n'_i} < \sum{n_i}$.  
Therefore, $(2m - 1) \times \max(\{n'_i\}) > \sum{n'_i}$ and $M' \not \in_\circlearrowleft P$.

If $n'_i < n_i$ for every $i\in \{1, \ldots, m-1\} \cup \{m+2, \ldots, 2m\}$, 
then $n'_m = n_m$ and $n'_{m+1} = n_{m+1}$ hold.

Since $n'_m + n'_{m+1} = n_m + n_{m+1} = n_1 > n'_1$, we do not have $n'_0 = 0, n'_1 = \cdots = n'_{m-1} = n'_m + n'_{m+1} = n'_{m+2}= \cdots = n'_{2m}$.

Also, we do not have $n'_i = 0, n'_{(i+1) \bmod (2m+1)} = \cdots = n'_{(i + m-1) \bmod (2m+1)} = n'_{(i+m) \bmod (2m+1)} + n'_{(i + m + 1) \bmod (2m+1)} = n'_{(i + m + 2) \bmod (2m+1)} = \cdots =  n'_{(i + 2m) \bmod (2m + 1)}$ for $i \in \{1, \ldots, m-1\} \cup \{m+2, \ldots, 2m\}$, since in this case, we have $n'_0 \ge n'_m$ and $n'_0 \ge n'_{m+1}$, which contradicts $n'_0 = n_0 = 0$ and $n'_{m}+ n'_{m+1} > 0$.

The case $n'_m = 0, n'_{m+1} = n'_{m+2} = \cdots = n'_{2m} + n'_0 = n'_1 = \cdots = n'_{m-1}$
also contradicts because $n_{m+1}'=n'_m + n'_{m+1} = n_m + n_{m+1} = n_{m+2}>n_{m+2}'$. In a similar way, $n'_{m+1} = 0, n'_{m+2} = n'_{m+3}= \cdots = n'_{2m} =  n'_0 + n'_1 = n'_2 = \cdots = n'_m$ contradicts, too.

Therefore, for $M\in P,$ every option $M' \in N$.


Finally, consider a position $M = (n_0, n_1,\ldots, n_{2m}) \not \in N$.

Note that except for $V \setminus \{v_i, v_{(i + m) \bmod (2m+1)}\} (0 \le i \le 2m),$ every set of five vertices can be a face of this ruleset since $2m + 1$ is a prime number.

Without loss of generality, we assume that $\min(\{n_i\}) = n_0$. We consider two cases: 

\begin{itemize}
\item[(i)] Without loss of generality, we assume that $ \min(\{n_1, n_2, \ldots, n_{2m}\}) = n_i(i \in \{1, 2, \ldots, m - 1\})$ or $\min(\{n_1, n_2, \ldots, n_{2m}\}) = n_{m}$, and we consider the first case.  
Let $n'_0 = n_0, n'_m = 0, n'_{2m} = n_i - n_0, $ and $n'_j = n_i$ for any $j \in \{1, 2, \ldots, m-1, m+1, \ldots, 2m-1\}$.

Then, $M' = (n'_0, n'_1, \ldots, n'_{2m})$ is an option of $M$ and $M' \in_\circlearrowleft P$. 

\item[(ii)] For the second case, we assume that $ \min(\{n_1, n_2, \ldots, n_{2m}\}) = n_m$.

\begin{itemize}
\item[(ii-a)]
If $\min(\{n_1,n_2, \ldots, n_{2m}\}\setminus\{n_m\}) = n_i(i \in \{1, 2, \ldots, 2m-1\} \setminus \{m, m+1\})$, then let $n'_0 = 0, n'_m = n_m, n'_{m+1} = n_i - n_m$ and $n'_j = n_i$ for any $j \in \{1, 2, \ldots, 2m+1\} \setminus \{0, m, m+1\}$. Then, $M' = (n'_0, n'_1, \ldots, n'_{2m})$ is an option of $M$ and $M' \in P$.

\item[(ii-b)] If $\min(\{n_1, n_2, \ldots, n_{2m}\}\setminus \{n_m\}) = n_{2m}$ and $n_0 + n_{2m} \leq \min(\{n_1, n_2, \ldots, n_{2m-1}\}\setminus \{n_m\})$, then let $n'_0 = n_0, n'_m = 0, n'_{2m} = n_{2m}$ and $n'_j = n_0 + n_{2m}$ for any $j \in \{1, \ldots, 2m - 1\} \setminus \{m\}$.
Then, $M' = (n'_0, n'_1, \ldots, n'_{2m})$ is an option of $M$ and $M' \in_\circlearrowleft P$. 

\item[(ii-c)]
Consider the case that $\min(\{n_1, n_2, \ldots, n_{2m}\} \setminus\{n_m\}) = n_{2m}$ and $n_0 + n_{2m} > \min(\{n_1, n_2, \ldots, n_{2m-1}\}\setminus \{n_m\})$. Let $a = \min(\{n_1, n_2, \ldots, n_{2m-1}\}\setminus \{n_m\})$. $n_0 \leq a, n_{2m} \leq a,$ and $ n_0 + n_{2m} > a$ hold.
\begin{itemize}
\item[(ii-c-1)]
If $n_i = a$ for $i \in \{1, 2, \ldots, 2m-1\} \setminus \{m, m+1\},$ then let $n'_0 = n_0, n'_m= 0, n'_{2m} = a - n_0$, and $n'_j = a$ for $j \in \{1, 2, \ldots, 2m-1\} \setminus \{m\}.$ Then, $M' = (n'_0, n'_1, \ldots, n'_{2m})$ is an option of $M$ and $M' \in_\circlearrowleft P$.
\item[(ii-c-2)]
If $n_{m+1} = a$, then let $n'_0 = a - n_{2m}, n'_m = 0, n'_{2m} =n_{2m},$ and $n'_j = a$ for any $j \in \{1, 2, \ldots, 2m - 1\} \setminus \{m\}$. Then, $M' = (n'_0, n'_1, \ldots,  n'_{2m})$ is an option of $M$ and $M' \in_\circlearrowleft P$.
\end{itemize}
\item[(ii-d)]If $\min(\{n_1, n_2, \ldots, n_{2m}\} \setminus \{n_m\}) = n_{m+1}$ and 
$n_m + n_{m+1} \leq \min(\{n_1, n_2, \ldots, n_{2m}\}\setminus \{n_m, n_{m+1}\})$, then let $n'_0 = 0, n'_m = n'_m, n'_{m+1} = n'_{m+1},$ and $n'_j = n_m + n_{m+1}$ for $j \in \{1, 2, \ldots, 2m\} \setminus \{m, m+1\}.$ Then, $M' = (n'_0, n'_1, \ldots, n'_{2m})$ is an option of $M$ and $M' \in P$.

\item[(ii-e)]
Consider the case that $\min(\{n_1, n_2, \ldots, n_{2m}\}) = n_{m+1}$ and $n_m + n_{m+1} > \min(\{n_1, n_2, \ldots, n_{2m}\}\setminus\{n_m,n_{m+1}\})$. Let $a = \min(\{n_1, n_2, \ldots, n_{2m}\}\setminus \{n_m, n_{m+1}\})$. $n_m \leq a, n_{m+1} \leq a$, and $n_m + n_{m+1} > a$ hold.
\begin{itemize}
\item[(ii-e-1)]If $n_i = a$ for $i \in \{2, 3, \ldots, 2m\} \setminus \{m, m+1\},$ then let $n'_0 = 0, n'_m = a - n_{m+1}, n'_{m + 1} = n_{m+1},$ and $n'_j = a$ for $j \in \{1, 2, \ldots, 2m\} \setminus \{m, m+1\}$. Then, $M' = (n'_0, n'_1, \ldots, n'_{2m})$ is an option of $M$ and $M' \in P$.
\item[(ii-e-2)]
If $n_1 = a$, then let $n'_0 = 0, n'_m =n_m, n'_{m+1} = a - n_m, $ and $n'_j = a$ for any $j \in \{ 1, 2, \ldots, 2m\} \setminus \{m, m+1\}.$ Then, $M' = (n'_0, n'_1, \ldots, n'_{2m})$ is an option of $M$ and $M' \in P$.
\end{itemize}
\end{itemize}
\end{itemize}

\end{proof}

\begin{corollary}
\label{gcn7125}
Consider a position $M = (n_0, n_1, n_2, n_3, n_4, n_5, n_6)$ in ${\rm ECN}(7_{\{1,2\}}, 5)$. $M$ is a $\mathcal{P}$-position  if and only if $M \in_\circlearrowleft P = \{(n_0, n_1, n_2, n_3, n_4, n_5, n_6) \mid  n_0 = 0, n_1 = n_2 = n_3 + n_4 = n_5 = n_6\}$.
\end{corollary}

Note that Theorem \ref{thm_ecn2m12m1} is also a generalization of  the result on ${\rm ECN}(5_{\{1\}}, 3)$. 

\begin{observation}
{\sc Extended circular nim} ${\rm ECN}(7_{\{1,2,3\}}, 2)$ is isomorphic to {\sc Moore's nim} ${\rm MN}(7, 2)$.
Also, {\sc Extended circular nim} ${\rm ECN}(7_{\{1,2,3\}}, 2)$ is isomorphic to ${\rm MN}(7,5)$.
\end{observation}

\begin{proof}
In ${\rm ECN}(7_{\{1,2,3\}}, 2)$, the player can choose any $2$ piles. Also, in ${\rm ECN}(7_{\{1,2,3\}}, 5)$, the player can choose any $5$ piles.
\end{proof}

From these discussions, we have characterized $\mathcal{P}$-positions in seven-pile {\sc extended circular nim} except for ${\rm ECN}(7_{\{1\}},2), {\rm ECN}(7_{\{1\}},3), {\rm ECN}(7_{\{1\}},5), {\rm ECN}(7_{\{1,2\}},2), \allowbreak {\rm ECN}(7_{\{1,2\}},3), {\rm ECN}(7_{\{1,2,3\}},3), {\rm ECN}(7_{\{1,2,3\}},4)$, and those which are isomorphic to them. Table \ref{matome4} summarizes these results.

\begin{table}[H]
\begin{center}
\begin{tabular}{c|c}
Ruleset & Result \\ \hline
${\rm ECN}(7_{\{1\}}, i)(2\leq i \leq 3)$ & Unsolved \\
${\rm ECN}(7_{\{1\}}, 4)$ & Shown in \cite{DHV22} \\
${\rm ECN}(7_{\{1\}}, 5)$ & Unsolved \\
${\rm ECN}(7_{\{2\}}, i)(2 \leq i \leq 5)$ & Isomorphic to ${\rm ECN}(7_{\{1\}}, i)$ \\
${\rm ECN}(7_{\{3\}}, i)(2 \leq i \leq 5)$ & Isomorphic to ${\rm ECN}(7_{\{1\}}, i)$ \\
${\rm ECN}(7_{\{1,2\}}, i)(2 \leq i \leq 3)$ & Unsolved \\
${\rm ECN}(7_{\{1,2\}}, 4)$ & Theorem \ref{gcn7124} \\
${\rm ECN}(7_{\{1,2\}}, 5)$ & Corollary \ref{gcn7125} \\
${\rm ECN}(7_{\{1,3\}}, i)(2 \leq i \leq 5)$ & Isomorphic to ${\rm ECN}(7_{\{1,2\}}, i)$ \\
${\rm ECN}(7_{\{2,3\}}, i)(2 \leq i \leq 5)$ & Isomorphic to${\rm ECN}(7_{\{1,2\}}, i)$ \\
${\rm ECN}(7_{\{1,2,3\}}, 2)$ & Isomorphic to ${\rm MN}(7, 2)$ \\
${\rm ECN}(7_{\{1,2,3\}}, i)(3 \leq i \leq 4)$ & Unsolved \\
${\rm ECN}(7_{\{1,2,3\}}, 5)$ & Isomorphic to ${\rm MN}(7,5)$ \\ \hline

\end{tabular}
\caption{Results of {\sc extended circular nim} with seven piles}
\label{matome4}
\end{center}
\end{table}


\section{{\sc Extended circular nim} with eight piles}
\label{sec_ecn8}
For {\sc extended circular nim} with eight piles, many cases are still open problems, but we can characterize $\mathcal{P}$-positions for some cases.

We consider ${\rm ECN}(8_S, i)(2 \leq i \leq 6)$.
As mentioned above, {\sc circular nim} with eight piles, or {\sc extended circular nim} ${\rm ECN}(8_{\{1\}}, i)$ is solved when $i \in  \{1, 6, 7, 8\}$.

When $S = \{2\}, $ the position $M = (n_0, n_1, n_2, n_3, n_4, n_5, n_6, n_7, n_8)$ can be considered as a disjunctive sum of positions $(n_0, n_2, n_4, n_6)$ and $(n_1, n_3, n_5, n_7)$ in {\sc circular nim} ${\rm CN}(4, i)$.
If $i = 2$, $M$ is a disjunctive sum of two positions in ${\rm CN}(4, 2)$, but any closed formula for Sprague-Grundy values of ${\rm CN}(4, 2)$ is unknown. Therefore, it is an open problem to characterize $\mathcal{P}$-positions in ${\rm ECN}(8_{\{2\}}, 2)$. If $i = 3, $ $M$ is a disjunctive sum of two positions in ${\rm MN}(4, 3)$. Therefore, by using Theorem \ref{thm:mooregrundy},  the set of $\mathcal{P}$-positions can be characterized. 
If $i \geq 4, $ $M$ is the disjunctive sum of positions  in ${\rm CN}(4, i)$. In this case, each player can choose all four piles which have  $n_0, n_2, n_4, n_6 $ tokens or all four piles which have $n_1, n_3, n_5, n_7$ tokens. Therefore, this position can be considered as a position $(n_0 + n_2 + n_4 + n_6, n_1 + n_3 + n_5 +n_7)$ in two-pile {\sc nim} and $M$
 is a  $\mathcal{P}$-position if and only if $n_0 + n_2 + n_4 + n_6 = n_1 + n_3 + n_5 +n_7$.

When $S = \{3\},$ ${\rm ECN}(8_{\{3\}}, i)$ is isomorphic to ${\rm ECN}(8_{\{1\}}, i)$. As shown in Figure \ref{81}, we have this isomorphism by exchanging names of piles.

\begin{figure}[tb]
\centering
\includegraphics[width = 5cm]{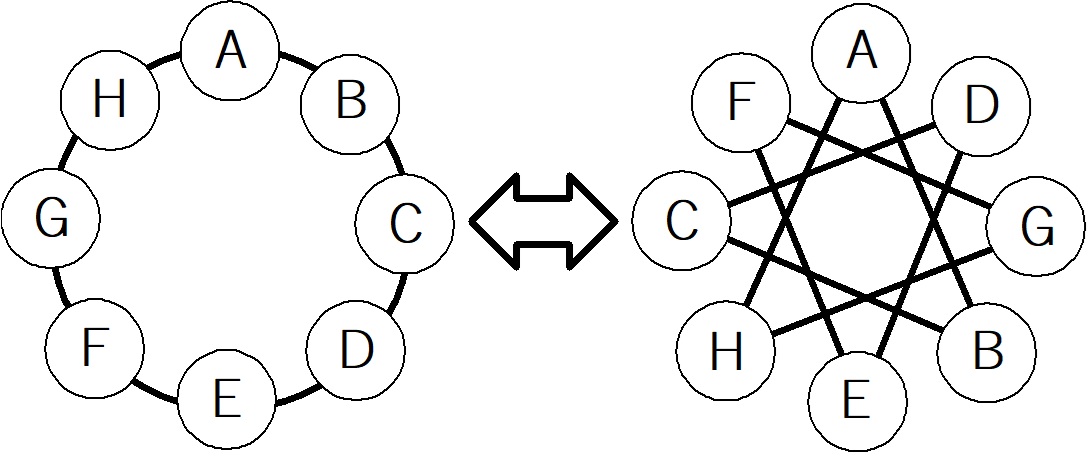}
\caption{${\rm ECN}(8_{\{1\}}, i)$ and ${\rm ECN}(8_{\{3\}}, i)$ are isomorphic}
\label{81}
\end{figure}


When $S = \{4\}$ and $i \geq 2$, ${\rm ECN} (8_{\{4\}}, i)$ can be considered as a position $(n_0+n_4, n_1 + n_5, n_2 + n_6, n_3 + n_7) $ in  four-pile {\sc nim} and $M$ is a $\mathcal{P}$-position if and only if $(n_0 + n_4) \oplus (n_1 + n_5) \oplus (n_2 + n_6) \oplus (n_3 + n_7) = 0$ .

When $S = \{1, 2\}, {\rm ECN}(8_{\{1,2\}},i)$ is unsolved for $2 \le i \le 6.$

For ${\rm ECN}(8_{\{1, 3\}}, 2),$ we can  obtain the following corollary from Theorem \ref{thm_2m2}.

\begin{corollary}
\label{thm8132}
Consider a position $M = (n_0, n_1, n_2, n_3, n_4, n_5, n_6, n_7)$ in ${\rm ECN}(8_{\{1,3\}}, 2)$. $M$ is a $\mathcal{P}$-position if and only if $M \in P = \{(n_0, n_1, n_2, n_3, n_4, n_5, n_6, n_7) \mid n_0 \oplus n_2 \oplus n_4 \oplus n_6 = n_1 \oplus n_3 \oplus n_5 \oplus n_7 = 0\}$.
\end{corollary}


\begin{theorem}
\label{thm8134}
Consider a position $M = (n_0, n_1, n_2, n_3, n_4, n_5, n_6, n_7)$ in ${\rm ECN}(8_{\{1,3\}}, 4)$. $M$ is a $\mathcal{P}$-position if and only if $M \in P = \{(n_0, n_1, n_2, n_3, n_4, n_5, n_6, n_7) \mid n_0 = n_4, n_1 = n_5, n_2 = n_6, n_3 = n_7\}$. 
\end{theorem}

\begin{proof}
In ${\rm ECN}(8_{\{1,3\}}, 4)$, the player can reduce at most one of $n_0, n_4$, at most one of $n_1, n_5$, at most one of $n_2, n_6,$ and at most one of $n_3, n_7$ in one move. Therefore, this ruleset is the selective sum of four positions $(n_0, n_4), (n_1, n_5), (n_2, n_6),$ and $(n_3, n_7)$ in two-pile {\sc nim}. Thus, from Theorems \ref{thm_nim} and \ref{thm:selsum}, $M$ is a $\mathcal{P}$-position if and only if $n_0 \oplus n_4 = n_1 \oplus n_5 =  n_2 \oplus n_6 = n_3 \oplus n_7 = 0,$ that is, $n_0 = n_4, n_1 = n_5, n_2 = n_6, n_3 = n_7.$

\end{proof}

For ${\rm ECN}(8_{\{1,3\}}, 6),$ we show the following more generalized theorem. 

\begin{theorem}
\label{thm_2m2m2}
Let $m (> 1)$ be an integer and assume that $h = 2^{m-1} - 1$. Consider a position $M = (n_0, n_1, \ldots, n_{2^m-1})$ in ${\rm ECN}((2^m)_{\{1, 3, \ldots, h\}},2^m-2)$. Then, $M$ is a $\mathcal{P}$-position if and only if $M \in P=\{(n_0, n_1, \ldots, n_{2^m-1}) \mid n_0 = n_2 = \cdots = n_{2^m-2} , n_1 = n_3 = \cdots = n_{2^m-1}\}$.
\end{theorem}

\begin{proof}
Note that every odd number is prime to $2^m$.
In ${\rm ECN}((2^m)_{\{1,3,\ldots, h\}}, 2^m-2)$, the player can reduce at most $2^{m-1}-1$ numbers of $n_0, n_2, \ldots, n_{2^m-2}$ and at most $2^{m-1}-1$ numbers of $n_1, n_3, \ldots, n_{2^m-1}$ in one move. Therefore, this ruleset is the selective sum of two positions $(n_0, n_2, \ldots, n_{2^m-2})$ and $ (n_1, n_3, \ldots, n_{2^m-1})$ in ${\rm MN}(2^{m-1}, 2^{m-1}-1)$. Thus, from Theorems \ref{thm_moo} and \ref{thm:selsum}, $M$ is a $\mathcal{P}$-position if and only if $n_0 = n_2 = \cdots = n_{2^m-2}$ and $n_1 = n_3 = \cdots = n_{2^m-1}$.
\end{proof}

\begin{corollary}
\label{thm8136}
Consider a position $M = (n_0, n_1, n_2, n_3, n_4, n_5, n_6, n_7)$ in ${\rm ECN}(8_{\{1,3\}}, 6)$. $M$ is a $\mathcal{P}$-position if and only if $M \in P = \{(n_0, n_1, n_2, n_3, n_4, n_5, n_6, n_7) \mid n_0 = n_2 = n_4 = n_6, n_1 = n_3 = n_5 = n_7\}$. 
\end{corollary}

Note that Theorem \ref{thm_2m2m2} is also a generalization of the result on ${\rm ECN}(4_{\{1\}}, 2)$.



When $S = \{1, 4\}$ and $2 \leq i \leq 4$, ${\rm ECN}(8_{\{1, 4\}}, i)$ has not been solved. For the case of $i \geq 5$, ${\rm ECN}(8_{\{1, 4\}}, i)$ is isomorphic to ${\rm ECN}(8_{\{1\}}, i)$.  

When $S = \{2, 3\},$ similar to the case of $S = \{3\}$, by exchanging the names of the piles, we can show that ${\rm ECN}(8_{\{2, 3\}}, i)$ is isomorphic to ${\rm ECN}(8_{\{1, 2\}}, i)$.

When $S = \{2, 4\}$, similar to the case of $S = \{2\}$, the position is a disjunctive sum of two positions $(n_0, n_2, n_4, n_6)$ and $(n_1, n_3, n_5, n_7)$ in ${\rm ECN}(4_{\{1,2\}}, \min(\{4,i\}))$.
If $i = 2$, then ${\rm ECN}(4_{\{1,2\}}, 2)$ is isomorphic to ${\rm MN}(4,2)$, but closed formula for Sprague-Grundy values of this ruleset is not known. Therefore, it is an open problem to characterize $\mathcal{P}$-positions in ${\rm ECN}(8_{\{2, 4\}}, 2).$ If $i \geq 3$, ${\rm ECN}(8_{\{2,4\}}, i)$ is isomorphic to ${\rm ECN}(8_{\{2\}}, i)$.

When $S = \{3, 4\}$, similar to the case of $S = \{3\}$, by exchanging the names of the piles, we can show that ${\rm ECN}(8_{\{3,4\}}, i)$ is isomorphic to ${\rm ECN}(8_{\{1,4\}}, i)$.

When $S = \{1, 2, 3\}$ and $2 \leq i \leq 5$, ${\rm ECN}(8_{\{1, 2, 3\}}, i)$ has not been solved. For the case of $i = 6$, we have the following theorem.

\begin{theorem}
\label{thmecn81236}

Consider a position $M = (n_0, n_1, n_2, n_3, n_4, n_5, n_6, n_7)$ in ${\rm ECN}(8_{\{1,2,3\}}, 6)$. $M$ is a $\mathcal{P}$-position if and only if $M \in_\circlearrowleft P = \{(n_0, n_1, n_2, n_3, n_4, n_5, n_6, n_7) \mid n_0 = n_2 = n_4 = n_6 = n_1 + n_3 + n_5 + n_7 \} \setminus \{(n_0, n_1, n_2, n_3, n_4, n_5, n_6, n_7) \mid n_0 = n_2 = n_4 = n_6 \text{ and } n_1 = n_3 = n_5 = n_7\} \cup \{(0, 0, 0, 0, 0, 0, 0, 0)\}$
\end{theorem}
\begin{proof}
Let $N = \{M=(n_0, n_1, n_2, n_3, n_4, n_5, n_6)\mid M \not \in_\circlearrowleft P\}$. The terminal position is in $P$.

Consider a non-terminal position $M \in_\circlearrowleft P$. Without loss of generality, we assume that $M = (n_0, n_1, n_2, n_3, n_4, n_5, n_6, n_7) \in P$.

There are three types of moves:
\begin{itemize}
    \item[(i)] Consider the case that the player selects $n_0, n_2, n_4, n_6$ or part of them to reduce.
    
    Let $M' = (n'_0, n_1, n'_2, n_3, n'_4, n_5, n'_6, n_7)(n'_0 + n'_2 + n'_4 + n'_6 < n_0 + n_2 + n_4 + n_6)$ is an option of $M$. If $n'_0 = n'_2 = n'_4 = n'_6$, then $n_1+  n_3 + n_5 + n_7 > n'_0$. Also, since $M \in P$ and $n_0 = n_2 = n_4 = n_6$, we do not have $n_1 = n_3 = n_5 = n_7$. Therefore, $M' \in N$.
\item[(ii)] Consider the case that the player selects  $n_1, n_3, n_5, n_7$ or part of them to reduce. 

Let $M' = (n_0, n'_1, n_2, n'_3, n_4, n'_5, n_6, n'_7)(n'_1 + n'_3 + n'_5 + n'_7 < n_1 + n_3 + n_5 + n_7)$ is an option of $M$. Since $n_0 = n_2 = n_4 = n_6 > n'_1 + n'_3 + n'_5 + n'_7$, $M' \in N$.
    
\item [(iii)]Consider the case that the player selects at least one and at most three of $n_0, n_2, n_4, n_6$ and at least one and at most three of $n_1, n_3, n_5, n_7$ to reduce. 

Let $M' = (n'_0, n'_1, n'_2, n'_3, n'_4, n'_5, n'_6, n'_7)(n'_0 + n'_2 + n'_4 + n'_6 < n_0 + n_2 + n_4 + n_6, n'_1 + n'_3 + n'_5 + n'_7 < n_1 + n_3 + n_5 + n_7)$ is an option of $M$ after such a move. Then, we do not have $n'_0 = n'_2 = n'_4 = n'_6$. Also, since $\max(n'_0, n'_2, n'_4, n'_6) = n_0 = n_1 + n_3 + n_5 + n_7 > n'_1 + n'_3 + n'_5 + n'_7$, we do not have $n'_1 = n'_3 = n'_5 = n'_7 = n'_0 + n'_2 + n'_4 + n'_6$. Therefore, $M'  \in N$. 
\end{itemize}

Finally, consider a position $M =  (n_0, n_1, n_2, n_3, n_4, n_5, n_6, n_7) \in N$.

If $n_0 = n_2 = n_4 = n_6 = n_1 + n_3 + n_5 + n_7$ and $n_1 = n_3 = n_5 = n_7$, then $M' = (n_1, n_1, 0, n_3, 0, n_5, 0, n_7)$ is an option of $M$ and $M' \in_\circlearrowleft P$.

Otherwise, without loss of generality, we assume $\min(\{n_0, n_2, n_4, n_6\}) < \min(\{n_1, n_3, n_5, n_7\})$.
\begin{itemize}
\item[(i)]
Consider the case that  $n_0 + n_2 + n_4 + n_6 \leq \min(\{n_1, n_3, n_5, n_7\})$. If we have $n_0 = n_2 = n_4 = n_6$, $M' = (n_0, n_0, n_2, 0, n_4, 0, n_6, 0)$ is an option of $M$ and $M' \in P$.
If we do not have $n_0 = n_2 = n_4 = n_6$, $M' = (n_0, n_0+n_2+n_4+n_6, n_2, n_0+n_2+n_4+n_6, n_4, n_0+n_2+n_4+n_6, n_6, n_0+n_2+n_4+n_6)$ is an option of $M$ and $M' \in_\circlearrowleft P$.
\item[(ii)]
Consider the case that $n_0 + n_2 + n_4 + n_6 > \min(\{n_1, n_3, n_5, n_7\})$. Without loss of generality, $n_1 = \min(\{n_1, n_3, n_5, n_7\})$. For the case of $n_0 = \min(\{n_0, n_2, n_4, n_6\}),$ since $n_1 > n_0 = \min(\{n_0, n_2, n_4, n_6\})$, we have $(n_0, n'_2, n'_4, n'_6)$ such that $n_0 + n'_2 + n'_4 + n'_6 = n_1, n'_2 \leq n_2, n'_4 \leq n_4, n'_6 \leq n_6,$ and not satisfy $n_0 = n'_2 = n'_4 = n'_6$. Therefore, $M' = (n_0, n_1, n'_2, n_1, n'_4, n_1, n'_6, n_1)$ is an option of $M$ and $M' \in_\circlearrowleft P$. 

Similarly, for the cases of $n_2 = \min(\{n_0, n_2, n_4, n_6\}), n_4 = \min(\{n_0, n_2, n_4, n_6\}),$ and $n_6 = \min(\{n_0, n_2, n_4, n_6\}),$ we have an option $M' \in_\circlearrowleft P$. 
\end{itemize}
\end{proof}

When $S = \{1, 2, 4\}$, ${\rm ECN}(8_{\{1, 2, 4\}}, 2)$ has not been solved. For the case of $3 \leq i \leq 6$, ${\rm ECN}(8_{\{1, 2, 4\}}, i)$ is isomorphic to ${\rm ECN}(8_{\{1,2\}}, i)$.

When $S = \{1, 3, 4\}$ and $2 \leq i \leq 4$, ${\rm ECN}(8_{\{1, 3,4\}}, i)$ has not been solved. For the case of $5 \leq i \leq 6$, ${\rm ECN}(8_{\{1, 3,4\}}, i)$ is isomorphic to ${\rm ECN}(8_{\{1,3\}}, i)$.

When $S = \{2, 3, 4\}$,  similar to the case of $S = \{3\}$, by exchanging the names of the piles, we can show that 
${\rm ECN}(8_{\{2, 3, 4\}}, i)$ is isomorphic to ${\rm ECN}(8_{\{1, 2, 4\}}, i)$.

When $S = \{1, 2, 3, 4\}$, if $i = 2$, ${\rm ECN}(8_{\{1, 2, 3, 4\}}, 2)$ is isomorphic to ${\rm MN}(8, 2)$. For the case of $3 \leq i \leq 6$, ${\rm ECN}(8_{\{1, 2, 3, 4\}}, 2)$ is isomorphic to ${\rm ECN}(8_{\{1, 2, 3\}}, i)$.

Table \ref{matome5} summarizes  results shown in this section.

\begin{table}[H]
\begin{center}
\begin{tabular}{c|c}
Ruleset & Result \\ \hline

${\rm ECN}(8_{\{1\}}, i)(2 \leq i \leq 5)$ & Unsolved \\
${\rm ECN}(8_{\{1\}}, 6)$ & Shown in \cite{DH13} \\

${\rm ECN}(8_{\{2\}}, i)(2 \leq i \leq 6)$ & Disjunctive sum of ${\rm CN}(4, \min(\{4, i\}))$ \\

${\rm ECN}(8_{\{3\}}, i)(2 \leq i \leq 6)$ & Isomorphic to ${\rm ECN}(8_{\{1\}}, i)$ \\

${\rm ECN}(8_{\{4\}}, i)(2 \leq i \leq 6)$ & Regarded as four-pile {\sc nim} \\

${\rm ECN}(8_{\{1, 2\}}, i)(2 \leq i \leq 6)$ & Unsolved \\

${\rm ECN}(8_{\{1, 3\}}, 2)$ & Corollary \ref{thm8132} \\
${\rm ECN}(8_{\{1, 3\}}, 3)$ & Unsolved \\
${\rm ECN}(8_{\{1, 3\}}, 4)$ & Theorem \ref{thm8134} \\
${\rm ECN}(8_{\{1, 3\}}, 5)$ & Unsolved \\
${\rm ECN}(8_{\{1, 3\}}, 6)$ & Corollary \ref{thm8136} \\

${\rm ECN}(8_{\{1, 4\}}, i)(2 \leq i \leq 4)$ & Unsolved \\
${\rm ECN}(8_{\{1, 4\}}, i)(5 \leq i \leq 6)$ & Isomorphic to ${\rm ECN}(8_{\{1\}}, i)$ \\

${\rm ECN}(8_{\{2, 3\}}, i)(2 \leq i \leq 6)$ & Isomorphic to ${\rm ECN}(8_{\{1, 2\}}, i)$ \\

${\rm ECN}(8_{\{2, 4\}}, i)(2 \leq i \leq 6)$ & Disjunctive sum of ${\rm ECN}(4_{\{1,2\}}, \min(\{4, i\}))$ \\

${\rm ECN}(8_{\{3, 4\}}, i)(2 \leq i \leq 6)$ & Isomorphic to ${\rm ECN}(8_{\{1, 4\}}, i)$ \\

${\rm ECN}(8_{\{1, 2, 3\}}, i)(2 \leq i \leq 5)$ & Unsolved \\
${\rm ECN}(8_{\{1, 2, 3\}}, 6)$ & Theorem \ref{thmecn81236} \\

${\rm ECN}(8_{\{1, 2, 4\}}, 2)$ & Unsolved \\
${\rm ECN}(8_{\{1, 2, 4\}}, i)(3 \leq i \leq 6)$ & Isomorphic to ${\rm ECN}(8_{\{1, 2\}}, i)$ \\

${\rm ECN}(8_{\{1, 3, 4\}}, i)(2 \leq i \leq 4)$ & Unsolved \\
${\rm ECN}(8_{\{1, 3, 4\}}, i)(5 \leq i \leq 6)$ & Isomorphic to ${\rm ECN}(8_{\{1, 3\}}, i)$ \\

${\rm ECN}(8_{\{2, 3, 4\}}, i)(2 \leq i \leq 6)$ & Isomorphic to ${\rm ECN}(8_{\{1, 2, 4\}}, i)$ \\

${\rm ECN}(8_{\{1, 2, 3, 4\}}, 2)$ & Isomorphic to ${\rm MN}(8, 2)$ \\
${\rm ECN}(8_{\{1, 2, 3, 4\}}, i)(3 \leq i \leq 6)$ & Isomorphic to ${\rm ECN}(8_{\{1, 2, 3\}}, i)$ \\ \hline

\end{tabular}
\caption{Results of {\sc extended circular nim} with eight piles}
\label{matome5}
\end{center}
\end{table}




\end{document}